\documentclass[a4paper]{article}
\usepackage{amsthm,amssymb,amsmath,enumerate,graphicx,psfrag}
\usepackage{algorithm}
\usepackage{algorithmic}
\usepackage{fullpage}
\usepackage{hyperref}

\newtheorem{definition}{Definition}
\newtheorem{claim}{Claim}

\newtheorem{theorem}[definition]{Theorem}

\newtheorem{lemma}[definition]{Lemma}
\newtheorem{conjecture}[definition]{Conjecture}

\usepackage{color}

\newcommand{\comment}[1]{}

\graphicspath{{figures/}}

\title{Exhaustive generation of $k$-critical $\mathcal H$-free graphs}
\author{Jan Goedgebeur\thanks{Jan Goedgebeur, 
\texttt{jan.goedgebeur@ugent.be},
Department of Applied Mathematics, Computer Science \& Statistics, Ghent University, Krijgslaan 281-S9, 9000 Ghent, Belgium}
\and Oliver Schaudt\thanks{Oliver Schaudt, 
\texttt{schaudto@uni-koeln.de},
Institut f\"ur Informatik,
Universit\"at zu K\"oln, K\"oln, Germany}}
\date{}

\begin{document}
\maketitle

\begin{abstract}
We describe an algorithm for generating all $k$-critical $\mathcal H$-free graphs, based on a method of Ho\`{a}ng et al. 
Using this algorithm, we prove that there are only finitely many $4$-critical $(P_7,C_k)$-free graphs, for both $k=4$ and $k=5$.
We also show that there are only finitely many $4$-critical graphs $(P_8,C_4)$-free graphs. For each case of these cases we also give the complete lists of critical graphs and vertex-critical graphs. These results generalize previous work by Hell and Huang,
and yield certifying algorithms for the $3$-colorability problem in the respective classes.

Moreover, we prove that for every $t$, the class of 4-critical planar $P_t$-free graphs is finite.
We also determine all 27 4-critical planar $(P_7,C_6)$-free graphs.

We also prove that every $P_{11}$-free graph of girth at least five is 3-colorable, and show that this is best possible by determining the smallest 4-chromatic $P_{12}$-free graph of girth at least five.
Moreover, we show that every $P_{14}$-free graph of girth at least six and every $P_{17}$-free graph of girth at least seven is 3-colorable.
This strengthens results of Golovach et al. 
\end{abstract}



\section{Introduction}

Given a graph $G$, a \emph{$k$-coloring} is a mapping $c : V(G) \to \{1,\ldots,k\}$ with $c(u) \neq c(v)$ for all edges $uv$ of $G$.
If a $k$-coloring exists for $G$, we call $G$ \emph{$k$-colorable}.
Moreover $G$ is called \emph{$k$-chromatic} if it is $k$-colorable, but not $(k-1)$-colorable.

The graph $G$ graph is called \emph{$k$-critical} if it is $k$-chromatic, but every proper subgraph of $G$ is $(k-1)$-colorable.
For example, the class of $3$-critical graphs equals the family of odd cycles. 
To characterize $k$-critical graphs is a notorious problem in graph theory.

To get a grip on this problem, it is common to consider graphs with restricted structure, as follows.
Let a graph $H$ and a number $k$ be given.
An \emph{$H$-free} graph is a graph that does not contain $H$ as an induced subgraph.
We say that a graph $G$ is \emph{$k$-critical $H$-free} if $G$ is $H$-free, $k$-chromatic, and every $H$-free proper subgraph of $G$ is $(k-1)$-colorable. 
If $\mathcal H$ is a set of graphs, then we say that a graph $G$ is $\mathcal H$-free if $G$ is $H$-free for each $H \in \mathcal H$.
The definition of a \emph{$k$-critical $\mathcal H$-free} graph is analogous.

A notion similar to critical graphs is that of $k$-\emph{vertex-critical} graphs: $k$-chromatic graphs whose every proper induced subgraph is $(k-1)$-colorable.
We define $k$-vertex-critical $H$-free and $k$-vertex-critical $\mathcal H$-free graphs accordingly.
Note that, unlike for critical graphs, the set of $k$-vertex-critical $\mathcal H$-free graphs equals the set of $\mathcal H$-free $k$-vertex-critical graphs.

We remark that every $k$-critical graph is $k$-vertex-critical.
Moreover, as noted by Ho\`ang et al.~\cite{HMRSV15}, there are finitely many $k$-critical $\mathcal H$-free graphs if and only if there are finitely many $k$-vertex-critical $\mathcal H$-free graphs, for any family of graphs $\mathcal H$.

The study of $k$-critical graphs in a particular graph class received a significant amount of interest in the past decade, which is partly due to the interest in the design of certifying algorithms.
Given a decision problem, a solution algorithm is called \emph{certifying} if it provides, together with the yes/no decision, a polynomial time verifiable certificate for this decision.
In case of $k$-colorability for $\mathcal H$-free graphs, a canonical certificate would be either a $k$-coloring or an induced subgraph of the input graph which is (a) not $k$-colorable and (b) of constant size.
However, assertion (b) can only be realized if there is a finite list of $(k+1)$-critical $\mathcal H$-free graphs.

Let us now mention some results in this line of research. 
From Lozin and Kami\'nski~\cite{K-L-col-g} and Kr\`al et al.~\cite{K-K-T-W-col-g} we know that the $k$-colorability problem remains NP-complete on $H$-free graphs, unless $H$ is the disjoint union of paths.
This motivates the study of graph classes in which some path is forbidden as induced subgraph.

Bruce et al.~\cite{BHS09} proved that there are exactly six 4-critical $P_5$-free graphs, where $P_t$ denotes the path on $t$ vertices.
Later, Maffray and Morel~\cite{MM12}, by characterizing the 4-vertex-critical $P_5$-free graphs, designed a linear time algorithm to decide 3-colorability of $P_5$-free graphs.
Randerath et al.~\cite{randerath_04} have shown that the only 4-critical $(P_6,C_3)$-free graph is the Gr\"otzsch graph.
More recently, Hell and Huang~\cite{hell_14} proved that there are four 4-critical $(P_6,C_4)$-free graphs.
They also proved that in general, there are only finitely many $k$-critical $(P_6,C_4)$-free graphs.

In a companion paper, we proved the following dichotomy theorem for the case of a single forbidden induced subgraph, answering questions of Golovach et al.~\cite{GJPS15}~and Seymour~\cite{SeyPriv}.

\begin{theorem}[Chudnovsky et al.~\cite{CGSZ15}]\label{thm:M2}
Let $H$ be a connected graph. 
There are finitely many $4$-critical $H$-free graphs if and only if $H$ is a subgraph of $P_6$. 
\end{theorem}

The main difficulty in the proof of the above theorem is to show that there are only finitely many 4-critical $P_6$-free graphs, namely 24.
A substantial step in this proof, in turn, is to show that there only finitely many 4-critical $(P_6,\mbox{diamond})$-free graphs.
After unsuccessfully trying to prove this by hand, we developed an algorithm to automatize the huge amount of case distinctions, based on a method recently proposed by Ho\`ang et al.~\cite{HMRSV15}. 

In the present paper, we thoroughly extend this algorithm in order to derive more characterizations of 4-critical $\mathcal H$-free graphs.
As a demonstration of the power of this algorithm, we prove the following results.
\begin{itemize}
	\item There are exactly 6 4-critical $(P_7,C_4)$-free graphs.
	\item There are exactly 17 4-critical $(P_7,C_5)$-free graphs.
	\item There are exactly 94 4-critical $(P_8,C_4)$-free graphs.
	\item There are exactly 27 4-critical planar $(P_7,C_6)$-free graphs. In addition, from a result of B\"ohme et al.~\cite{BMSS04} we derive that for every $t$ there are only finitely many 4-critical planar $P_t$-free graphs.
	\item Every $P_{11}$-free graph of girth at least five is 3-colorable and there is a 4-chromatic $P_{12}$-free graph of girth 5. 
	\item Every $P_{14}$-free graph of girth at least six is 3-colorable.
	\item Every $P_{17}$-free graph of girth at least seven is 3-colorable.
\end{itemize}

Our results extend and/or strengthen previous results of Hell and Huang~\cite{hell_14} and Golovach et al.~\cite{golovach2014coloring}.

Besides these results, some of which we think are infeasible to prove by hand, we see the algorithm as the main contribution of our paper.
Its modular design allows to easily implement new expansion rules, which would make it more powerful and in turn might lead to a proof of several open problems in this line of research.
To this end, we also mention two cases that were out of reach with the current algorithm, but where we have a good feeling that there is only a finite set of obstructions. 

In the next section we propose a number of lemmas that give necessary conditions for $k$-critical graphs.
Our generation algorithm, which we also present in the next section, is built on these lemmas.

In Section~\ref{sect:testing_results} we prove the above mentioned results, and also discuss some results from the literature that we verified.


\section{A generic algorithm to find all $k$-critical $\mathcal H$-free graphs}

We build upon a method recently proposed by Ho\`ang et al.~\cite{HMRSV15}.
With this method they have shown that there is a finite number of 5-critical $(P_5,C_5)$-free graphs. 

The idea is to use necessary conditions for a graph to be critical to generate all critical graphs. 
The algorithm then performs all remaining case distinctions automatically.
In order to deal with more advanced cases, we need to thoroughly alter the approach of Ho\`ang et al.~\cite{HMRSV15}.
We remark that the algorithm presented below is, moreover, a substantial strenghtening of that used in the proof of Theorem~\ref{thm:M2}, and this strenghtening is necessary in order to derive the results of the present paper.

\subsection{Preparation}

In this section we prove a number of lemmas which will later be used as expansion rules for our generation algorithm.
Although each of them is straightforward, they turn out to be very useful for our purposes.

We use the following notation. The set $N_G(v)$ denotes the neighborhood of a vertex $v$ in $G$. If it is clear from the context which graph is meant, we abbreviate this to $N(v)$. The graph $G|U$ denotes the subgraph of $G$ induced by the vertex subset $U \subseteq V(G)$.
Moreover, for a vertex subset $U \subseteq V(G)$ we denote by $G-U$ the induced subgraph $G|(V(G)\setminus U)$ of $G$.

Let $G$ be a $k$-colorable graph.
The \emph{$k$-hull} of $G$, which we denote $G_k$, is the graph obtained from $G$ by making two vertices $u$ and $v$ adjacent if and only if there is no $k$-coloring of $G$ under which $u$ and $v$ receive the same color.
Clearly $G_k$ is a supergraph of $G$ without loops, and $G_k$ is $k$-colorable.

It is a folklore fact that a $k$-critical graph cannot contain two distinct vertices $u$ and $v$ with $N(u) \subseteq N(v)$.
The following observation is a proper generalization of this fact, and we proved it in~\cite{CGSZ15}.

\begin{lemma}[Chudnovsky et al.~\cite{CGSZ15}]\label{lem:color-trick}
Let $G=(V,E)$ be a $k$-vertex-critical graph and let $U,W$ be two non-empty disjoint vertex subsets of $G$.
Let $H:=(G-U)_{k-1}$.
If there exists a homomorphism $\phi : G|U \mapsto H|W$, then $N_G(u)\setminus U \not\subseteq N_H(\phi(u))$ for some $u \in U$.
\end{lemma}


We make use of Lemma~\ref{lem:color-trick} in the following way.
Assume that $G'$ is a $k$-vertex-critical graph and $G$ is a $(k-1)$-colorable induced subgraph of $G'$.
Then pick two disjoint vertex subsets $U$ and $W$ of $G$, both non-empty, and let $H:=(G-U)_{k-1}$.
Assume that there is a homomorphism $\phi : G|U \mapsto H|W$ with $N_G(u)\setminus U \subseteq N_H(\phi(u))$ for each $u \in U$.
Then we know that there is a vertex $x \in V(G')\setminus V(G)$ adjacent to some $u \in U$ but non-adjacent to $\phi(u)$, in the graph $G'$.
Also $x$ is non-adjacent to $\phi(u)$ in $(G'-U)_{k-1}$.

Recall that any $k$-critical graph $G$ must have minimum degree at least $k-1$.
As otherwise, any $(k-1)$-coloring of $G-u$ can be extended to a $(k-1)$-coloring of $G$, where $u$ is some vertex in $G$ of degree at most $k-2$. 
The following is an immediate strengthening.

\begin{lemma}\label{lem:min_degree}
Let $G$ be a $k$-vertex-critical graph and $u \in V(G)$.
Then in any $(k-1)$-coloring of $G-u$, the set $N_G(u)$ receives $k-1$ distinct colors. 
\end{lemma}

We make use of Lemma~\ref{lem:min_degree} in the following way.
Assume that $G$ is a $(k-1)$-colorable graph that is an induced subgraph of some $k$-vertex-critical graph $G'$.
Suppose that there is a vertex $u$ such that there is no $(k-1)$-coloring of $G-u$ in which the set $N_G(u)$ receives $k-1$ distinct colors. 
Let us say that $\ell$ is the maximum number of distinct colors that the set $N_G(u)$ can recieve in a $(k-1)$-coloring of $G-u$.
Then there must be some vertex $v \in N_{G'}(u)\setminus V(G)$ such that there is a $(k-1)$-coloring of $G'|(V(G-u)\cup \{v\})$ in which the set $N_G(u)\cup\{v\}$ receives $\ell+1$ distinct colors. 
\medskip

We also need the following fact which is folklore: every cutset of a critical graph contains at least two non-adjacent vertices, where a \emph{cutset} is a vertex subset whose removal increases the number of connected components of the graph.

\begin{lemma}\label{lem:cutset}
Let $G$ be a $k$-critical graph, and let $X$ be a clique of $G$. Then $G-X$ is a connected graph.
\end{lemma}

In fact, we only need that a $k$-critical graph does not have a cutvertex, that is, a cutset of size 1.
We use Lemma~\ref{lem:cutset} as follows.
Given a $(k-1)$-colorable graph $G$ that is an induced subgraph of some $k$-vertex-critical graph $G'$.
Suppose that there is a cutvertex $u$ in $G$.
Then there must be some vertex $v \in V(G')\setminus V(G)$ with at least one neighbor in $V(G)\setminus \{u\}$. 

The next lemma is custom-made for the case of $4$-critical graphs.

\begin{lemma}\label{lem:crowns}
Let $G$ be a 4-vertex-critical graph.
Suppose that there is an induced cycle $C$ where all vertices are of degree three.
Then $C$ has odd length, and there is a 3-coloring of $G-V(C)$ for which every member of the set $(\bigcup_{c \in V(C)}N(c)) \setminus V(C)$ receives the same color.
\end{lemma}
\begin{proof}
Consider the graph $C_k$, for some $k \ge 3$, where each vertex $v$ is equipped with a two-element list $L(v) \subseteq \{1,2,3\}$.
It is an easy exercise to show that this graph has a valid list coloring, unless $k$ is odd and all lists are identical.

Consequently, given any 3-coloring $c$ of $G-V(C)$, we may extend this coloring to a 3-coloring of $G$ unless $C$ is of odd length and $c$ is constant on the set $(\bigcup_{c \in V(C)}N(c)) \setminus V(C)$.
\end{proof}

We use Lemma~\ref{lem:crowns} as follows.
Given a $3$-colorable graph $G$ that is an induced subgraph of some $4$-vertex-critical graph $G'$.
Suppose that there is an induced cycle $C$ in $G$ where all vertices are of degree three.
Moreover, suppose that $C$ has even length, or $C$ has odd length and in every 3-coloring of $G-V(C)$ the set $(\bigcup_{c \in V(C)}N(c)) \setminus V(C)$ is not monochromatic.
Then there must be some vertex $v \in V(G')\setminus V(G)$ with at least one neighbor in $V(C)$.

The next lemma tells us from which graphs we have to start our enumeration algorithm.

\begin{lemma}\label{lem:startgraphs}
Every $k$-critical $P_t$-free graph different from $K_k$ contains one of the following graphs as induced subgraph:
\begin{itemize}
\item an odd hole $C_{2s + 1}$, for $2 \le s \le \lfloor (t-1)/ 2 \rfloor$, or
\item an odd antihole $\overline{C_{2s + 1}}$, for $3 \le s \le k - 1$.
\end{itemize}
\end{lemma}
\begin{proof}
It follows from the Strong Perfect Graph Theorem~\cite{chudnovsky_06} that every $k$-critical graph different from $K_k$ must contain an odd hole or anti-hole as an induced subgraph.
The statement of the lemma thus follows from the following facts: 
\begin{enumerate}[(a)]
	\item the cycle $C_{2s + 1}$ contains an induced $P_{2s}$,
  \item the antihole $\overline{C_{2k+1}}$ has clique number $k$, and
  \item $\overline{C_5}$ is isomorphic to $C_5$.
\end{enumerate}
\end{proof}

\subsection{The enumeration algorithm}


We use Algorithm~\ref{algo:init-algo} below to enumerate all $\mathcal H$-free $k$-critical graphs.
In order to keep things short, we use the following conventions for a $(k-1)$-colorable graph $G$.

We call a pair $(u,v)$ of distinct vertices for which $N_G(u) \subseteq N_{(G-u)_{k-1}}(v)$ \emph{similar vertices}.
Similarly, we call a 4-tuple $(u,v,u',v')$ of distinct vertices with $uv, u'v' \in E(G)$ such that $N_G(u)\setminus \{v\} \subseteq N_{(G-\{u,v\})_{k-1}}(u')$ and $N_G(v)\setminus \{u\} \subseteq N_{(G-\{u,v\})_{k-1}}(v')$ \emph{similar edges}.
Finally, we define \emph{similar triangles} in an analogous fashion.

Recall that a diamond is the graph obtained from $K_4$ by removing one edge.
We define \emph{similar diamonds} in complete analogy to similar triangles.

Let $u$ be a vertex of $G$ for which, in every $(k-1)$-coloring of $G-u$, the set $N_G(u)$ receives at most $k-2$ distinct colors.
Then we call $u$ a \emph{poor vertex}.

Let $C$ be an induced cycle in $G$ such that every vertex of $C$ has degree three.
We say that $C$ is a \emph{weak cycle} if $C$ is of even length or, if $C$ has odd length, there is no 3-coloring of $G-V(C)$ for which every member of the set $(\bigcup_{c \in V(C)}N(c)) \setminus V(C)$ receives the same color.

\begin{algorithm}[h]
\caption{Generate $\mathcal H$-free $k$-critical graphs}
\label{algo:init-algo}
  \begin{algorithmic}[1]
	\STATE let $\mathcal F$ be an empty list
	\STATE Construct($K_k$) \ // i.e.\ perform Algorithm~\ref{algo:construct}
	\FORALL{graphs $G$ mentioned in Lemma~\ref{lem:startgraphs}}
		\STATE Construct($G$) \ // i.e.\ perform Algorithm~\ref{algo:construct}
	\ENDFOR
	\STATE Output $\mathcal F$
  \end{algorithmic}
\end{algorithm}

\begin{algorithm}[ht!]
\caption{Construct(Graph $G$)}
\label{algo:construct}
  \begin{algorithmic}[1]
		\IF{$G$ is $\mathcal H$-free AND not generated before} \label{line:isocheck}
			\IF{$G$ is not $(k-1)$-colorable}
				\IF{$G$ is $k$-critical $\mathcal H$-free} \label{line:k-critical}
					\STATE add $G$ to the list $\mathcal F$
				\ENDIF	
			\ELSE
				\IF{$G$ contains similar vertices $(u,v)$}
					\FOR{every graph $H$ obtained from $G$ by attaching a new vertex $x$ and incident edges in all possible ways, such that $ux \in E(H)$, but $vx \notin E((H-u)_{k-1})$}
						\STATE\label{line:dominatedvertex} Construct($H$)
					\ENDFOR
				\ELSIF{$G$ contains a poor vertex $u$}\label{line:smallvertex}
					\FOR{every graph $H$ obtained from $G$ by attaching a new vertex $x$ and incident edges in all possible ways, such that $ux \in E(H)$ and the maximum number of distinct colors the set $N_H(u)$ recieves in some $(k-1)$-coloring of $H-u$ properly increased}
						\STATE Construct($H$)
					\ENDFOR
				\ELSIF{$G$ contains similar edges $(u,v,u',v')$}\label{line:dominatededge}	
					\FOR{every graph $H$ obtained from $G$ by attaching a new vertex $x$ and incident edges in all possible ways, such that $rx \in E(H)$ and $r'x \notin E((H-\{u,v\})_{k-1})$ for some $r \in \{u,v\}$}
						\STATE Construct($H$) 
					\ENDFOR
				\ELSIF{$G$ contains similar triangles $(u,v,w,u',v',w')$}\label{line:dominatedtriangle}
					\FOR{every graph $H$ obtained from $G$ by attaching a new vertex $x$ and incident edges in all possible ways, such that $rx \in E(H)$ and $r'x \notin E((H-\{u,v,w\})_{k-1})$ for some $r \in \{u,v,w\}$}
						\STATE Construct($H$)
					\ENDFOR						
				\ELSIF{$G$ contains similar diamonds $(u,v,w,x,u',v',w',x')$}\label{line:dominateddiamond}
					\FOR{every graph $H$ obtained from $G$ by attaching a new vertex $y$ and incident edges in all possible ways, such that $ry \in E(H)$ and $r'y \notin E((H-\{u,v,w,x\})_{k-1})$ for some $r \in \{u,v,w,x\}$}
						\STATE Construct($H$)
					\ENDFOR						
				\ELSIF{$k=4$ and $G$ contains a weak cycle $C$}
					\FOR{every graph $H$ obtained from $G$ by attaching a new vertex $x$ with a neighbor in $C$ and incident edges in all possible ways}
						\STATE Construct($H$) \label{line:weakcycle}
					\ENDFOR						
				\ELSIF{$G$ contains a cutvertex $u$} 
					\FOR{every graph $H$ obtained from $G$ by attaching a new vertex $x$ adjacent to some member of $V(G) \setminus \{u\}$ and incident edges in all possible ways}
						\STATE Construct($H$) \label{line:clique}
					\ENDFOR
				\ELSE 
					\FOR{every graph $H$ obtained from $G$ by attaching a new vertex $x$ and incident edges in all possible ways}
						\STATE Construct($H$)
					\ENDFOR		
				\ENDIF					
			\ENDIF	
		\ENDIF	  
  \end{algorithmic}
\end{algorithm}

We now prove that Algorithm~\ref{algo:init-algo} is correct.

\begin{theorem}\label{lem:diamond-algo-correct}
Assume that Algorithm~\ref{algo:init-algo} terminates, and outputs the list of graphs $\mathcal F$.
Then $\mathcal F$ is the list of all $k$-critical $\mathcal H$-free graphs.
\end{theorem}
\begin{proof}
In view of lines~\ref{line:isocheck} and~\ref{line:k-critical} of Algorithm~\ref{algo:construct}, it is clear that all graphs of $\mathcal F$ are $k$-critical $H$-free.
So, it remains to prove that $\mathcal F$ contains all $k$-critical $\mathcal H$-free graphs.
To see this, we first prove the following claim.

\begin{claim}\label{clm:invariant}
For every $k$-critical $\mathcal H$-free graph $F$ other than $K_k$, Algorithm~\ref{algo:construct} applied to some graph $F'$ from Lemma~\ref{lem:startgraphs} generates an isomorphic copy of an induced subgraph of $F$ with $i$ vertices for every $|V(F')| \le i \le |V(F)|$.
\end{claim}
We prove this by induction, as an invariant of our algorithm.
Due to Lemma~\ref{lem:startgraphs}, we know that $F$ contains some $F'$ as an induced subgraph, so the claim holds for $i=|V(F')|$.

So assume that the claim is true for some $i$ with $|V(F')| \le i \le |V(F)|-1$. 
Let $G$ be an induced subgraph of $F$ with $|V(G)|=i$ such that an isomorphic copy of $G$, say $G'$ is generated by the algorithm.
We may choose $G'$ such that it is not pruned by the isomorphism check on line~\ref{line:isocheck}.
In order to save notation, let us assume that $G'=G$.

First assume that $G$ contains similar vertices, and that the similar vertices $(u,v)$ are considered by the algorithm.
Then, by Lemma~\ref{lem:color-trick}, $N_F(u) \not\subseteq N_{(F-u)_{k-1}}(v)$.
Hence, there is some vertex $x \in V(F) \setminus V(G)$ which is adjacent to $u$ in $F$, but not to $v$ in $(F-u)_{k-1}$.
Following the statement of line~\ref{line:dominatedvertex}, Construct($F|(V(G) \cup \{x\})$) is called.
We omit the discussion of the lines~\ref{line:dominatededge},~\ref{line:dominatedtriangle}, and~\ref{line:dominateddiamond} as they are analogous. 

So assume that none of the criteria checked earlier apply to $G$, but the algorithm finds a poor vertex $u$ in $G$. 
Then, by~Lemma~\ref{lem:min_degree}, there is some vertex $x \in V(F) \setminus V(G)$ adjacent to $u$ such that the maximum number of distinct colors the set $N_{H}(u)$ recieves in some $(k-1)$-coloring of $H-u$ increased compared to $G-u$, where $H:=F|(V(G) \cup \{x\})$. 
Following the statement of line~\ref{line:smallvertex}, Construct($H$) is called.

Now assume that none of the criteria checked earlier apply to $G$, but we have $k=4$ and $G$ contains a weak cycle.
Say $C$ is the weak cycle considered by the algorithm.
By Lemma~\ref{lem:crowns}, $F$ does not contain a weak cycle, and so there must be some vertex $x \in V(F)\setminus V(G)$ with a neighbor in $C$.
Hence, Construct($F|(V(G) \cup \{x\})$) is called in line~\ref{line:weakcycle}.

Otherwise if the algorithm finds a cutvertex $u$ of $G$, there must be some vertex $x \in V(F)\setminus V(G)$ which is adjacent to some vertex in $V(G) \setminus \{u\}$.
This follows from Lemma~\ref{lem:cutset}. Following the statement of line~\ref{line:clique}, Construct($F|(V(G) \cup \{x\})$) is called.

Finally, if none of the above criteria apply to $G$, the algorithm attaches a new vertex to $G$ in all possible ways, and calls Construct for all of these new graphs.
Since $|V(F)|>|V(G)|$, among these graphs there is some induced subgraph of $F$, and of course this graph has $i+1$ vertices.
This completes the proof of Claim~\ref{clm:invariant}.

Given that the algorithm terminates, Claim~\ref{clm:invariant} implies that $\mathcal F$ must contain all $k$-critical $\mathcal H$-free graphs.
\end{proof}

We implemented this algorithm in C with some further optimizations (see Section~\ref{sect:optimisations}). We used the program \verb|nauty|~\cite{nauty-website, mckay_14} to make sure that no isomorphic graphs are accepted. More specifically, we use \verb|nauty| to compute a canonical form of the graphs. We maintain a list of the canonical forms of all non-isomorphic graphs which were generated so far and only accept a graph if it was not generated before (in which case its canonical form is added to the list).

Our program does indeed terminate in several cases. More details about this can be found in Section~\ref{sect:testing_results}. The source code of the program can be downloaded from~\cite{criticalpfree-site} and in the Appendix we describe how we extensively tested the correctness of our implementation.


\subsection{Optimizations and implementation details}
\label{sect:optimisations}

In this section we describe implementation details and additional optimizations which significantly speed up the program.

\subsubsection{Priority of operations}
The order or priority in which the next expansions are determined in Algorithm~\ref{algo:construct} are vital for the termination of the program. For most cases, the optimal order is as described in Algorithm~\ref{algo:construct}: i.e.\ first test if the graph contains similar vertices, if this is not the case, then test if it contains small vertices, if this is not the case then test if it contains similar edges, etc.

Even within the same category of similar \textit{elements}, the intermediate graphs generated by Algorithm~\ref{algo:construct} typically contain multiple similar elements. Choosing the right similar element is also important for the efficiency and termination of the algorithm. In most cases it is best to destroy the similar element with the smallest degree.

We also tried to extend every similar element once without iterating any further and then choosing the similar element with the least number of generated children and extending it recursively. But this was a lot slower and did not significantly reduce the number of graphs generated in the cases we investigated, except when generating critical graphs with a given minimal girth (see Section~\ref{sect:girth}).

\subsubsection{Testing of properties}

It is important to test $\mathcal H$-freeness, colorability and isomorphism in the right order. 
In our case, most generated graphs are $k$-colorable but not $\mathcal H$-free, and not too many isomorphic copies are generated.
In view of this, it is best to test these properties in the following order:
\begin{enumerate}
\item $\mathcal H$-freeness
\item $k$-colorability
\item Isomorphism
\end{enumerate}

In most cases the routine which tests if a graph contains an induced $P_t$ is a bottleneck.
As nearly all generated graphs contain an induced  $P_t$ (for the values of $t$ which are within reach of the program), the following heuristic helps a lot. Whenever an induced $P_t$ was found in a graph, we store it. When we test the next graph, we first test if one of the previously stored $P_t$'s from a previous graph (with the same number of vertices) is still an induced path in the current graph and only if this is not the case, we exhaustively search for an induced $P_t$. 

Experiments show that is optimal to store approximately 20 such previous $P_t$'s and in about 90\% of the cases this heuristic allows to reject the generated graph since an induced $P_t$ was found. 

\subsubsection{Limiting expansions}

The algorithm connects the new vertex $v_n$ to a set of vertices $S \subseteq \{v_0,...,v_{n-1}\}$. When connecting $v_n$ to $S$ yields a graph which is not $k$-colorable, connecting $v_n$ to $S' \supseteq S$ will yield a graph which is also not $k$-colorable and not $(k+1)$-critical, so these expansions can be skipped. Note that we cannot apply this optimization if we are looking for vertex-critical graphs. 

The same optimization can also be applied if connecting $v_n$ to $S$ yields a graph with a cycle of length smaller than $g$ or a non-planar graph when searching for graphs with girth at least $g$ or planar graphs, respectively.


\section{Results}
\label{sect:testing_results}

This section describes the main results obtained with our implementation of Algorithm~\ref{algo:init-algo}. In the Appendix we describe how we tested the correctness of our implementation.

The adjacency lists of all new critical graphs from this section can be found in Appendix~2 and these graphs can also be downloaded from the \textit{House of Graphs}~\cite{hog} at \url{http://hog.grinvin.org/Critical}

\subsection{Verification of previously known results}
\label{sect:previous_results}

As a correctness test and to demonstrate the strength of the approach we verified the following characterizations which were known before in the literature.
\begin{itemize}
\item There are six 4-critical $P_5$-free graphs~\cite{BHS09}.
\item There are eight 5-critical $(P_5,C_5)$-free graphs~\cite{HMRSV15}.
\item The Gr\"otzsch graph is the only 4-critical $(P_6,C_3)$-free graph~\cite{randerath_04}.
\item There are four 4-critical $(P_6,C_4)$-free graphs~\cite{hell_14}.
\end{itemize}
In each of the above cases our program terminates in a few seconds.

\subsection{4-critical $(P_r,C_s)$-free graphs}

It is known~\cite{CGSZ15} that there is an infinite family of 4-critical $P_7$-free graphs.
However, a careful observation shows that all members of this family contain a $C_k$ for all $k=3,4,5$, and are $C_\ell$-free for $\ell=6,7$.
This motivates the study of 4-critical $(P_7,C_k)$-free graphs when $k=3,4,5$, since there might be only finitely many of these.
We can solve the cases of $k=4$ and $k=5$ using Algorithm~\ref{algo:init-algo}.
Moreover, we can also solve the $(P_8,C_4)$-free case.

\begin{theorem}\label{thm:N3P7Ck}
The following assertions hold.
\begin{enumerate}[(a)]
	\item There are exactly 17 4-critical $(P_7,C_4)$-free graphs.
	\item There are exactly 94 4-critical $(P_8,C_4)$-free graphs.
	\item There are exactly 6 4-critical $(P_7,C_5)$-free graphs.
\end{enumerate}
\end{theorem}
\begin{proof}[Proof of Theorem~\ref{thm:N3P7Ck}.]
Algorithm~\ref{algo:init-algo} terminates in the $(P_7,C_4)$-free case and outputs 17 4-critical graphs.
It also terminates in the $(P_8,C_4)$-free case and outputs 94 4-critical graphs.
Hence, by Theorem~\ref{lem:diamond-algo-correct}, (a) and (b) both hold.

In the $(P_7,C_5)$-free case we need to use a slightly modified version of Algorithms~\ref{algo:init-algo} and~\ref{algo:construct}.
Instead of just graphs, the algorithms now consider pairs $P=(G,\mbox{triples}(P))$ where $G$ is a graph and for each $(u,v,x) \in \mbox{triples}(P)$, $u$, $v$, and $x$ are distinct vertices of $G$.
In Algorithm~\ref{algo:init-algo}, we call Construct$(P)$ for all pairs of the form $(G,\emptyset)$ where $G$ is a graph from Lemma~\ref{lem:startgraphs}.

At the beginning of Algorithm~\ref{algo:construct} where we consider some pair $P=(G,\mbox{triples}(P))$ we check whether, for each triple $(u,v,x) \in \mbox{triples}(P)$, there is some $(k-1)$-coloring of $G-u$ where $v$ receives the same color as $x$.
If this is not the case, then we discard the pair $P$ and return.
Since we need to seperately consider distinct pairs $P,P'$ even if their graphs $G,G'$ are isomorphic, we do not perform the isomorphism test in line~\ref{line:isocheck} anymore.

In all subsequent lines where usually Construct$(H)$ is called, we instead call Construct$(P')$, where $P'=(H,\mbox{triples}(P))$.
Note that triples$(P')$ equals triples$(P)$ here.

The only exception is the call of Construct$(H)$ in line~\ref{line:dominatedvertex}.
Let us say the algorithm considers the pair $(u,v)$ of similar vertices, and the vertex $x$ is the new vertex added to the graph $G$.
Then the algorithm calls Construct$(P')$ in line~\ref{line:dominatedvertex}, where $P'=(H,\mbox{triples}(P)\cup\{(u,v,x)\})$.

To see that this is correct, it suffices to prove the following altered version of Claim~\ref{clm:invariant}.

\begin{claim}\label{clm:invariant-advanced}
Let $F$ be a $k$-critical $\mathcal H$-free graph other than $K_k$. 
The modified version of Algorithm~\ref{algo:construct} applied to the pair $(F',\emptyset)$, where $F'$ is some graph from Lemma~\ref{lem:startgraphs}, generates a pair $P^i=(G^i,\mbox{triples}(P^i))$ with $|V(G^i)|=i$ 
for every $|V(F')| \le i \le |V(F)|$ such that following assertions hold.
\begin{enumerate}[(a)]
	\item There is an injective homomorphism $\phi^i$ from $G^{i}$ to $F$ such that, for each triple $(u,v,x) \in \mbox{triples}(P^i)$, there is some $(k-1)$-coloring of $F-\phi^i(u)$ where $\phi^i(v)$ and $\phi^i(x)$ receive the same color. 
	\item The graph $G^{i}$ is an induced subgraph of $G^{i+1}$, $\mbox{triples}(P^i) \subseteq \mbox{triples}(P^{i+1})$, and $\phi^{i+1}|_{V(G^i)} \equiv \phi^i$ for all $|V(F')| \le i \le |V(F)|-1$.
\end{enumerate}
\end{claim}
We prove this by induction, as an invariant of our algorithm.
Due to Lemma~\ref{lem:startgraphs}, we know that $F$ contains some $F'$ as an induced subgraph, so the claim holds for $i=|V(F')|$.

So assume that the claim is true for some $i$ with $|V(F')| \le i \le |V(F)|-1$. 
Let $P^i=(G^i,\mbox{triples}(P^i))$ and $\phi^i$ be as desired.
First assume that $G^i$ contains similar vertices $(u,v)$.
By Lemma~\ref{lem:color-trick}, $N_F(\phi^i(u)) \not\subseteq N_{(F-\phi^i(u))_{k-1}}(\phi^i(v))$.
Hence, there is some vertex $x \in V(F) \setminus \mbox{im}~\phi^i$ which is adjacent to $\phi^i(u)$ in $F$, but not to $\phi^i(v)$ in $(F-\phi^i(u))_{k-1}$.
Here, $\mbox{im}~\phi^i$ denotes the image of $\phi^i$.
Thus, we may safely define $\phi^{i+1}$ and the pair $P^{i+1}=(G^{i+1},\mbox{triples}(P^{i+1}))$ as follows.
\begin{enumerate}[(a)]
	\item $V(G^{i+1}) = V(G^i) \cup \{y\}$ and $G^{i+1}|V(G^i)=G^i$, where $y$ is the new vertex added by the algorithm in line~\ref{line:dominatedvertex},
	\item $\phi^{i+1}|_{V(G^i)} \equiv \phi^i$, $\phi^{i+1}(y)=x$,
	\item $\phi^{i+1}$ is an injective homomorphism from $G^{i+1}$ to $F$, and
	\item $P^{i+1}=P^i\cup (u,v,y)$.
\end{enumerate}
By construction, for each triple $(u,v,x) \in \mbox{triples}(P^{i+1})$ there is some $(k-1)$-coloring of $F-\phi^{i+1}(u)$ where $\phi^{i+1}(v)$ and $\phi^{i+1}(x)$ receive the same color.
Moreover, Construct$(P^{i+1})$ is called in line~\ref{line:dominatedvertex}.

If $G^i$ does not contain similar vertices, we proceed in complete analogy.
The only difference is that we have to set $P^{i+1}=P^i$.
\medskip

The remainder of the proof of Claim~\ref{clm:invariant-advanced} is literally the same as that of Claim~\ref{clm:invariant}.
\end{proof}

These 4-critical $(P_7,C_4)$-free and $(P_7,C_5)$-free graphs are shown in Figure~\ref{fig:animals_N3P7C4} and~\ref{fig:animals_N3P7C5}, respectively.
The adjacency lists of these graphs can be found in Appendix~2.

We also determined that there are exactly 35 4-vertex-critical $(P_7,C_4)$-free graphs, 164 4-vertex-critical $(P_8,C_4)$-free graphs, and 27 4-vertex-critical $(P_7,C_5)$-free graphs (details on how we obtained these graphs can be found in the Appendix). 

The Tables~\ref{table:counts_animals_N3P7C4},~\ref{table:counts_animals_N3P8C4}, and~\ref{table:counts_animals_N3P7C5} give an overview of the counts of the 4-critical and 4-vertex-critical graphs mentioned in Theorem~\ref{thm:N3P7Ck}.

\begin{figure}[h!t]
\centering
\includegraphics[width=.15\textwidth]{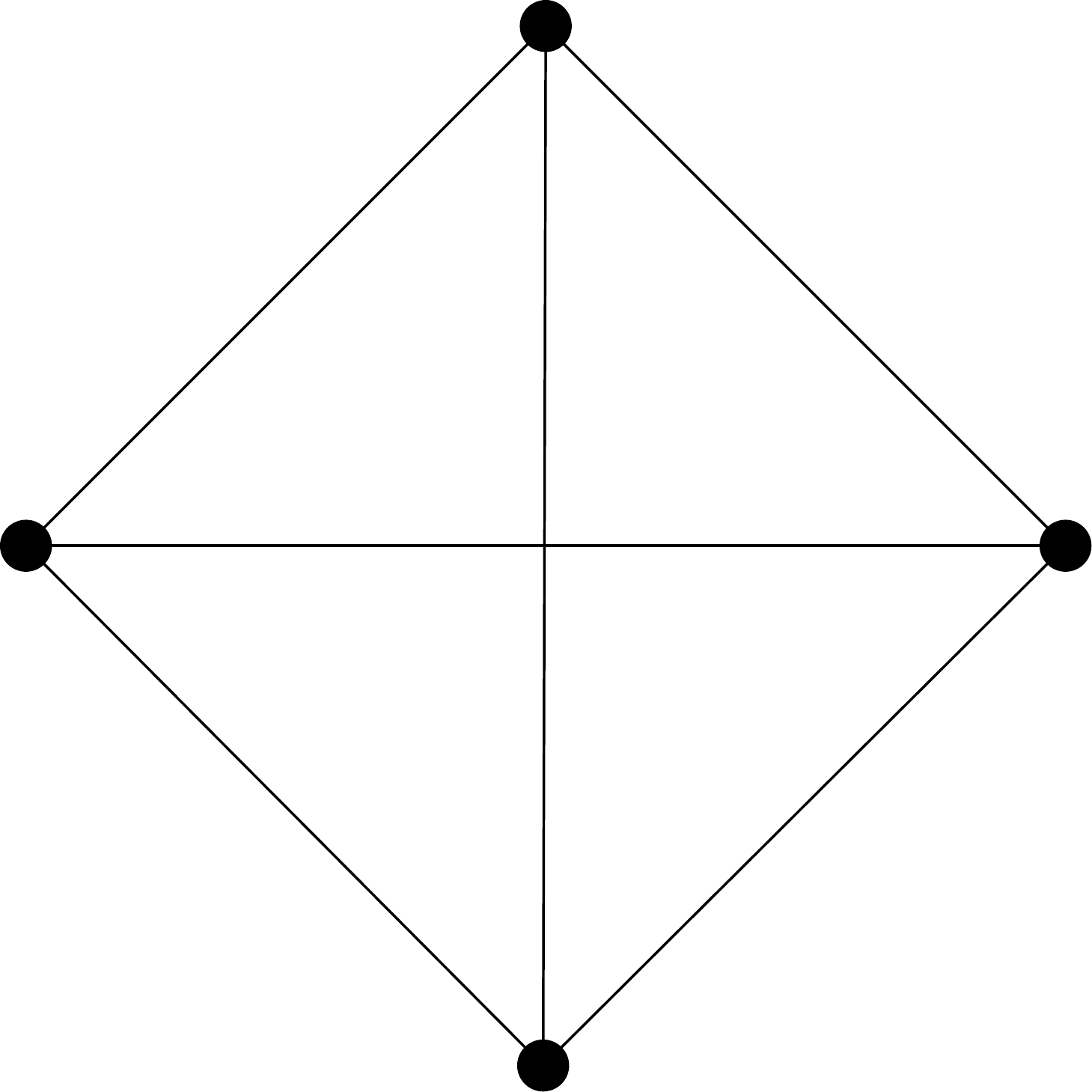}\ \ 
\includegraphics[width=.15\textwidth]{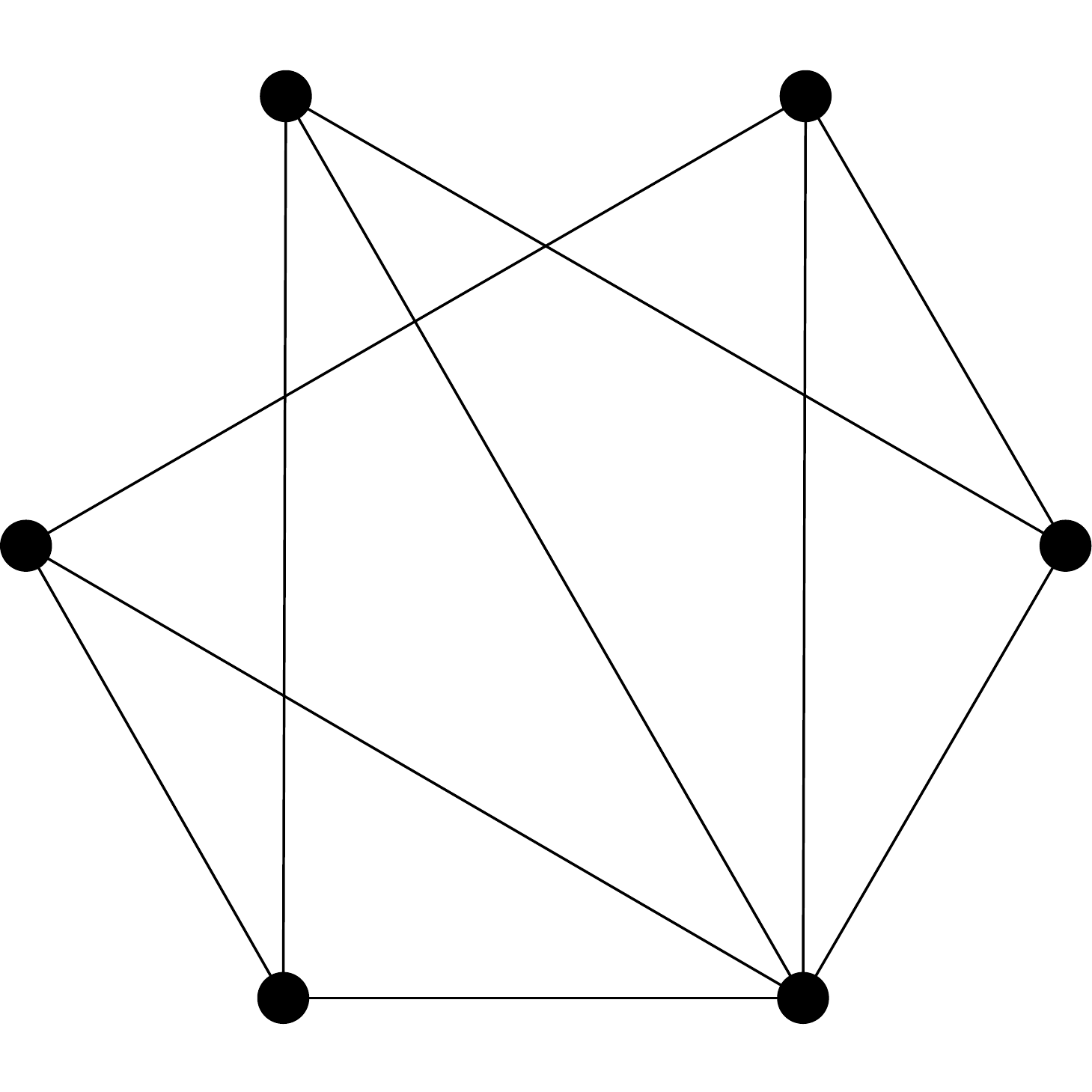}\ \ 
\includegraphics[width=.15\textwidth]{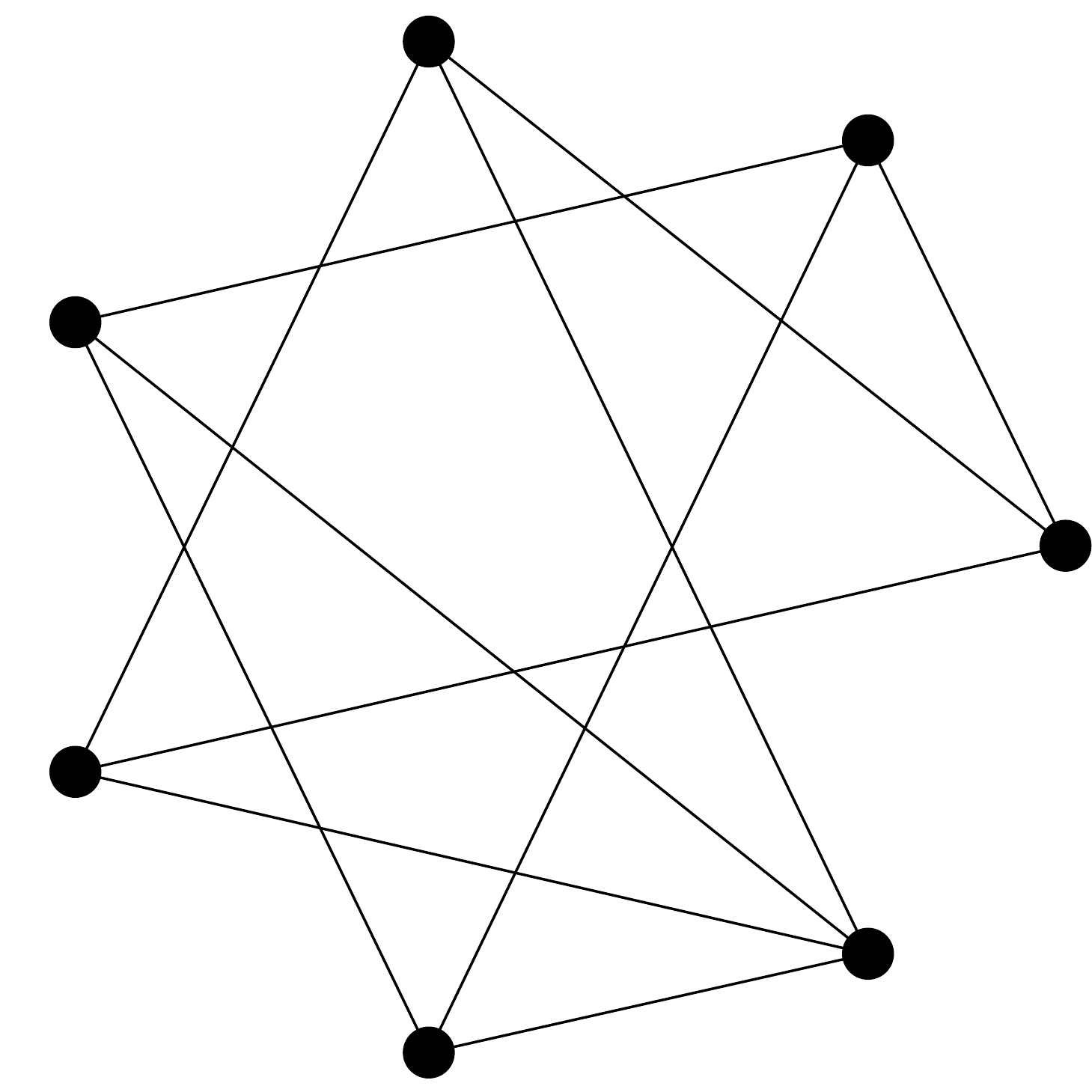}\ \ 
\includegraphics[width=.15\textwidth]{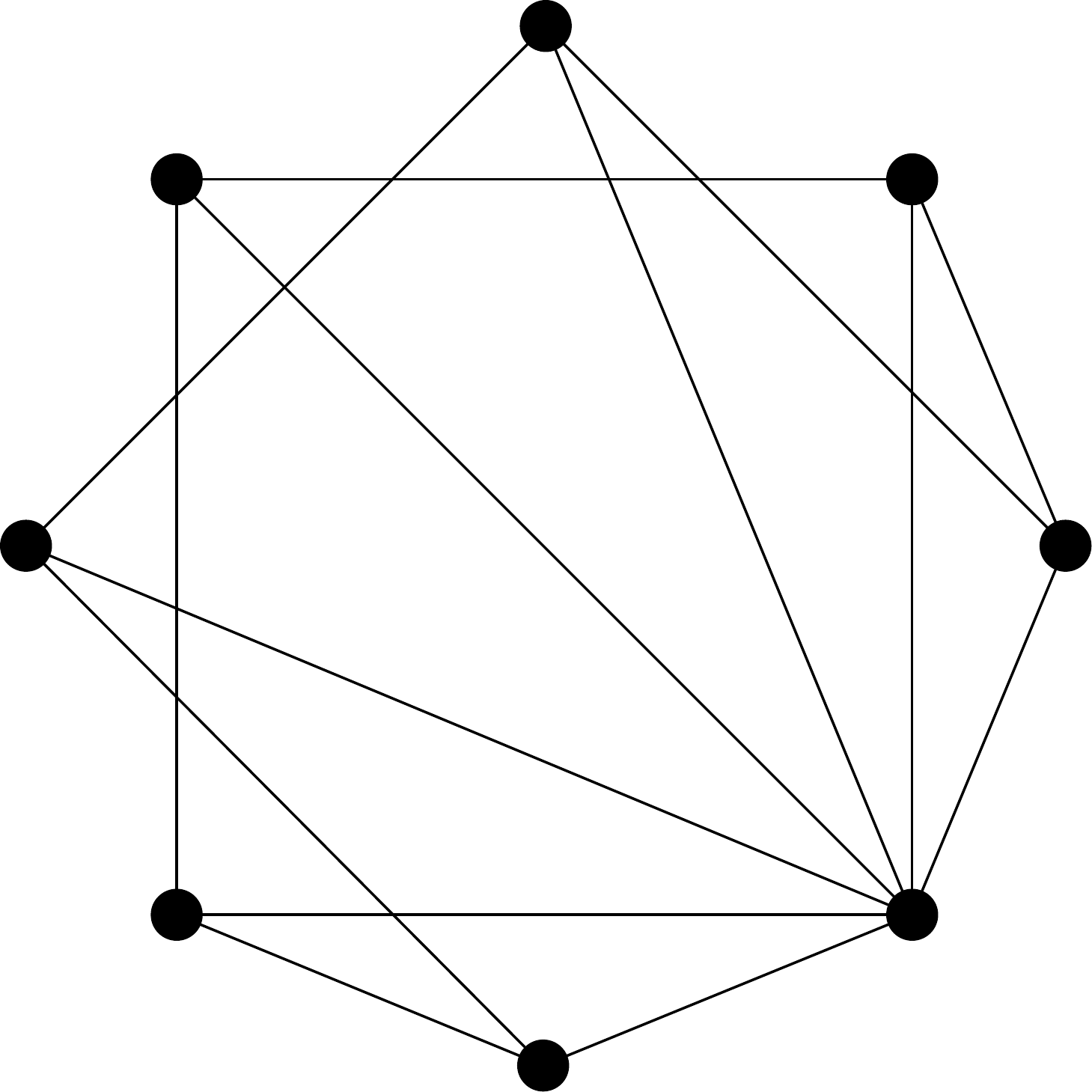}\ \ 
\includegraphics[width=.15\textwidth]{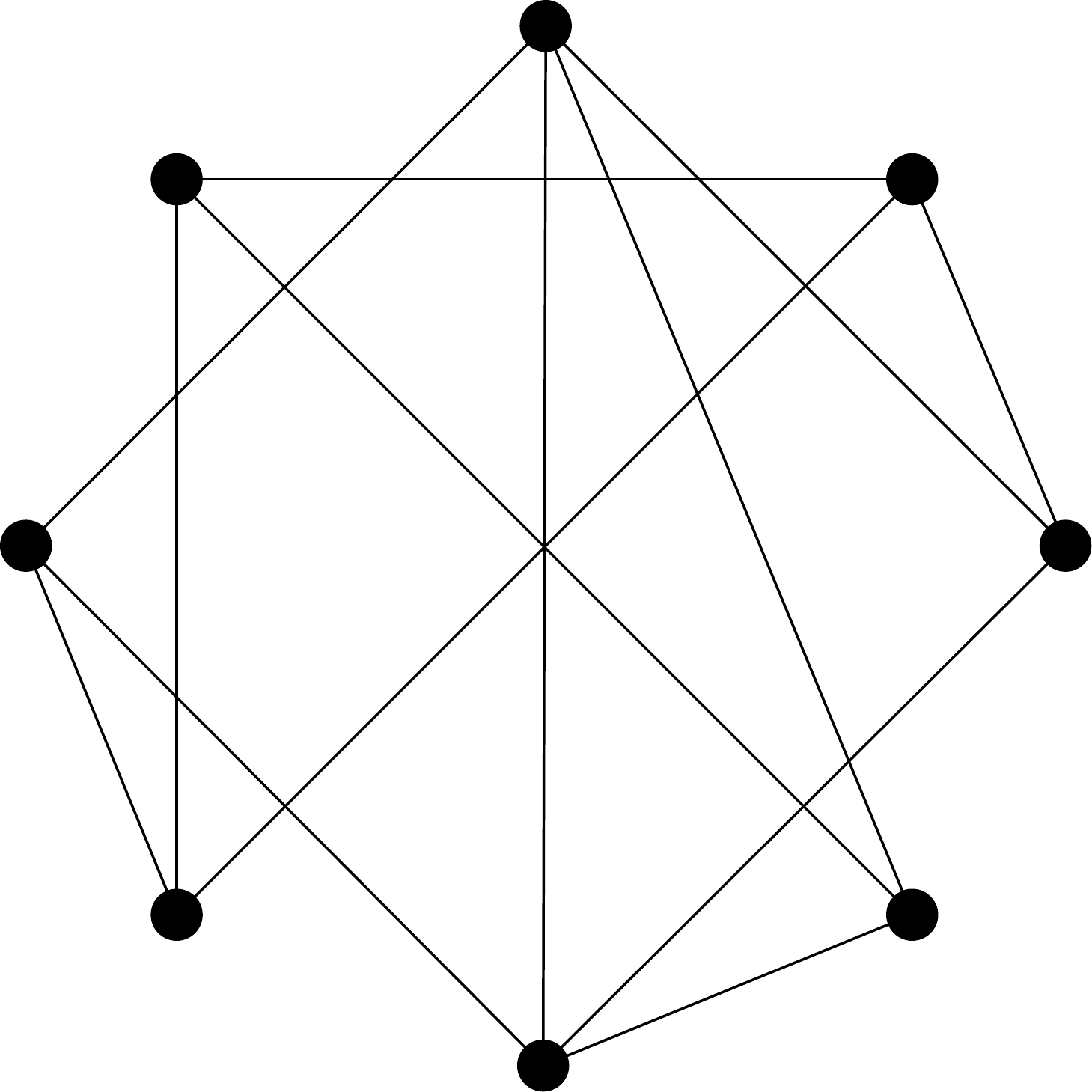}\ \ 
\includegraphics[width=.15\textwidth]{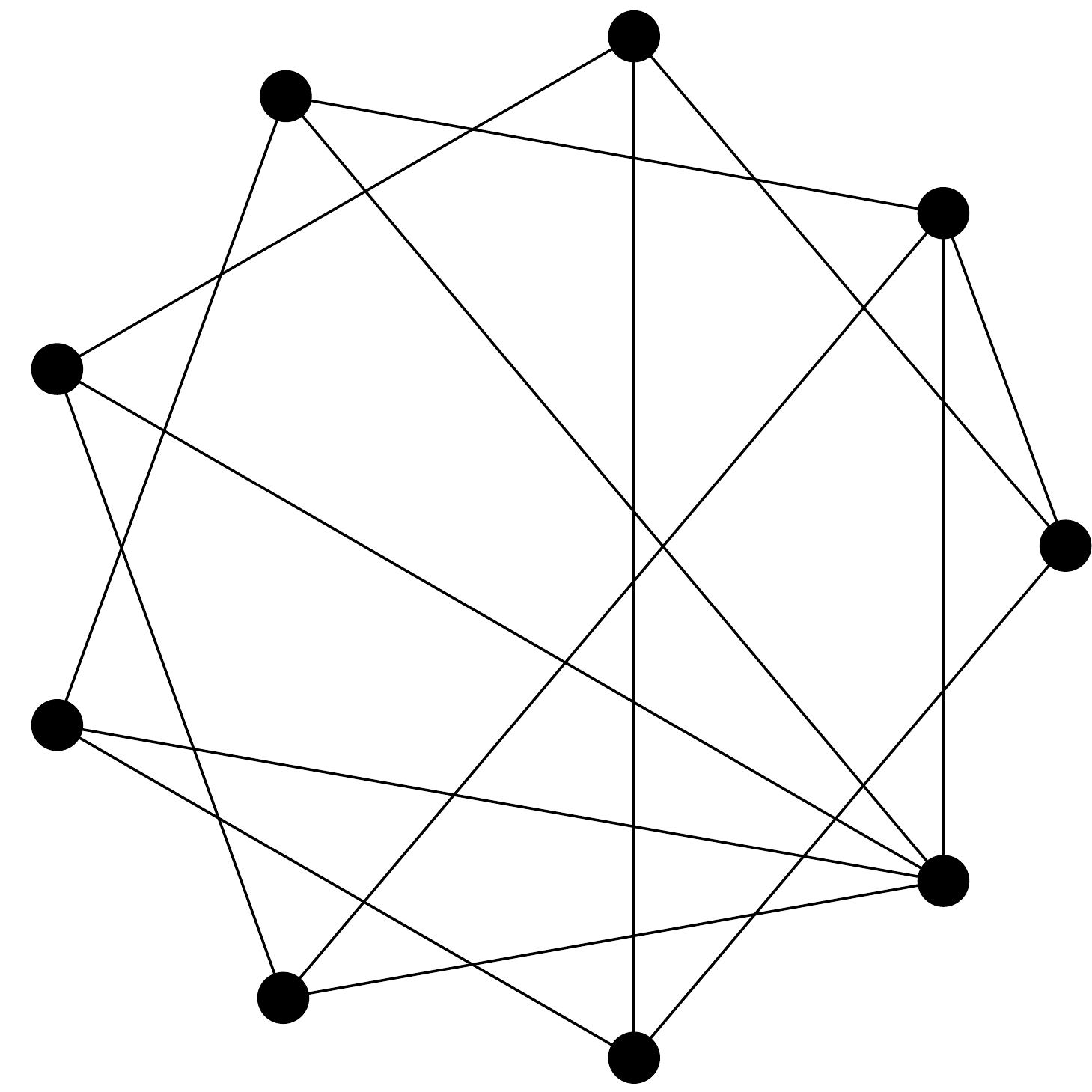}

\bigskip

\includegraphics[width=.15\textwidth]{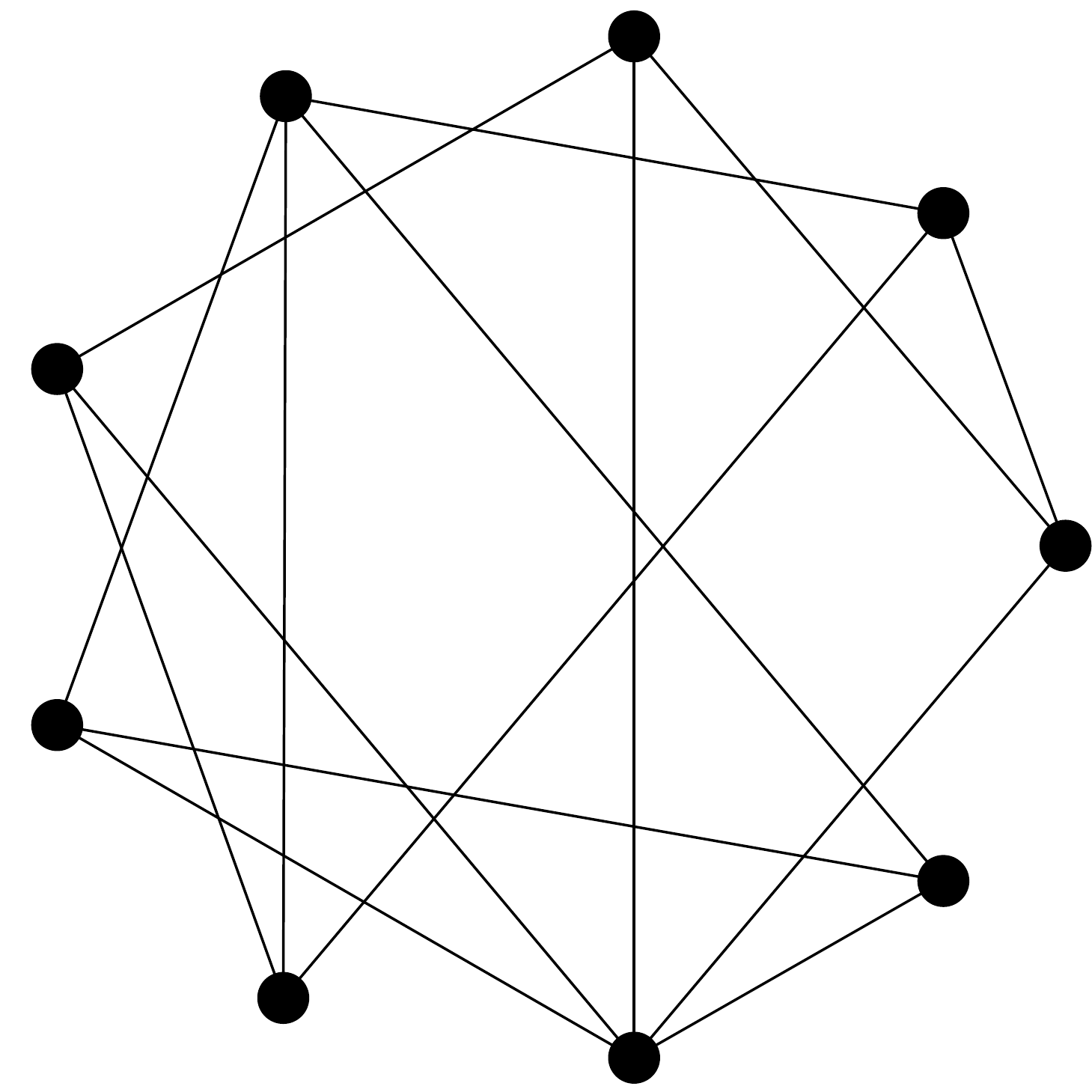}\ \ 
\includegraphics[width=.15\textwidth]{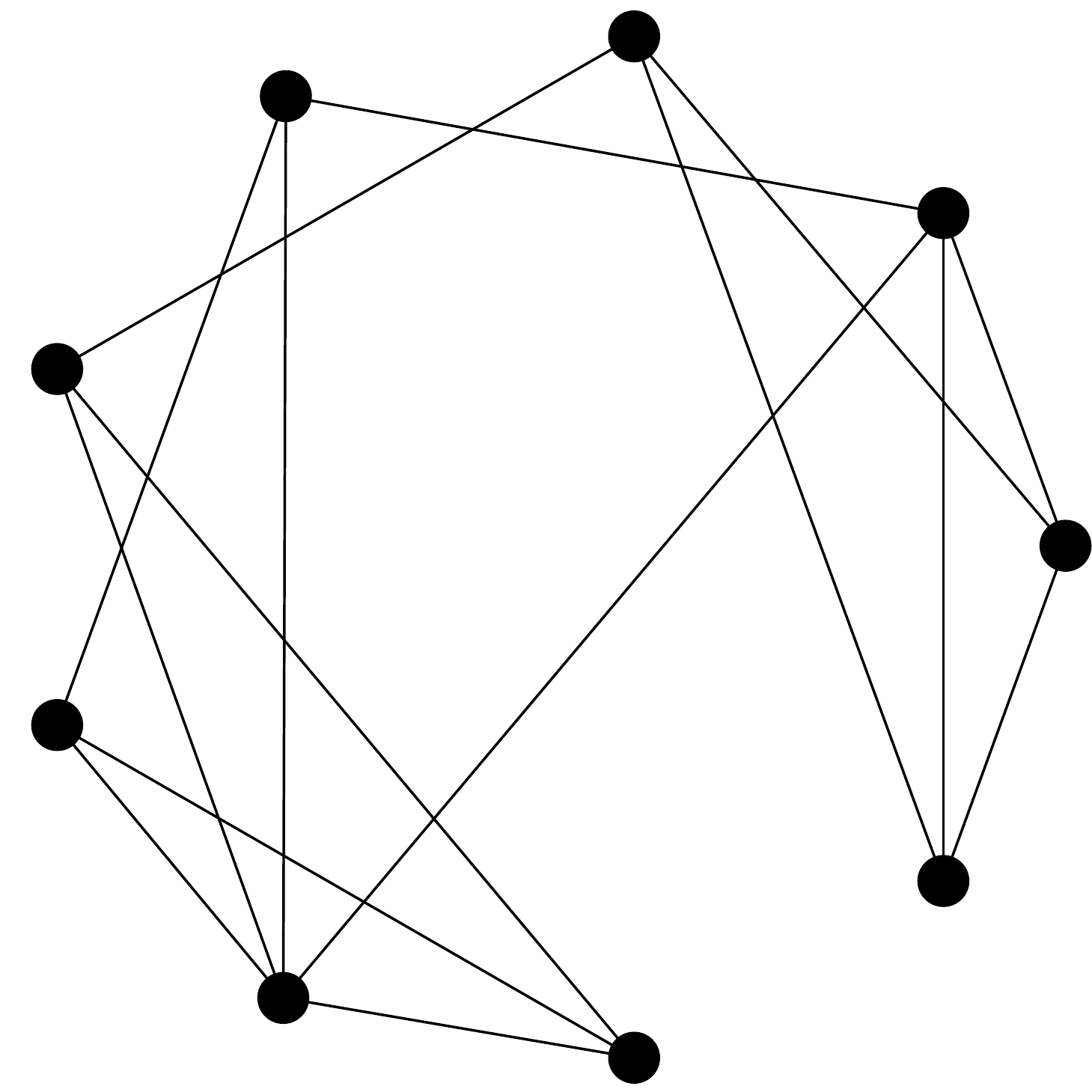}\ \ 
\includegraphics[width=.15\textwidth]{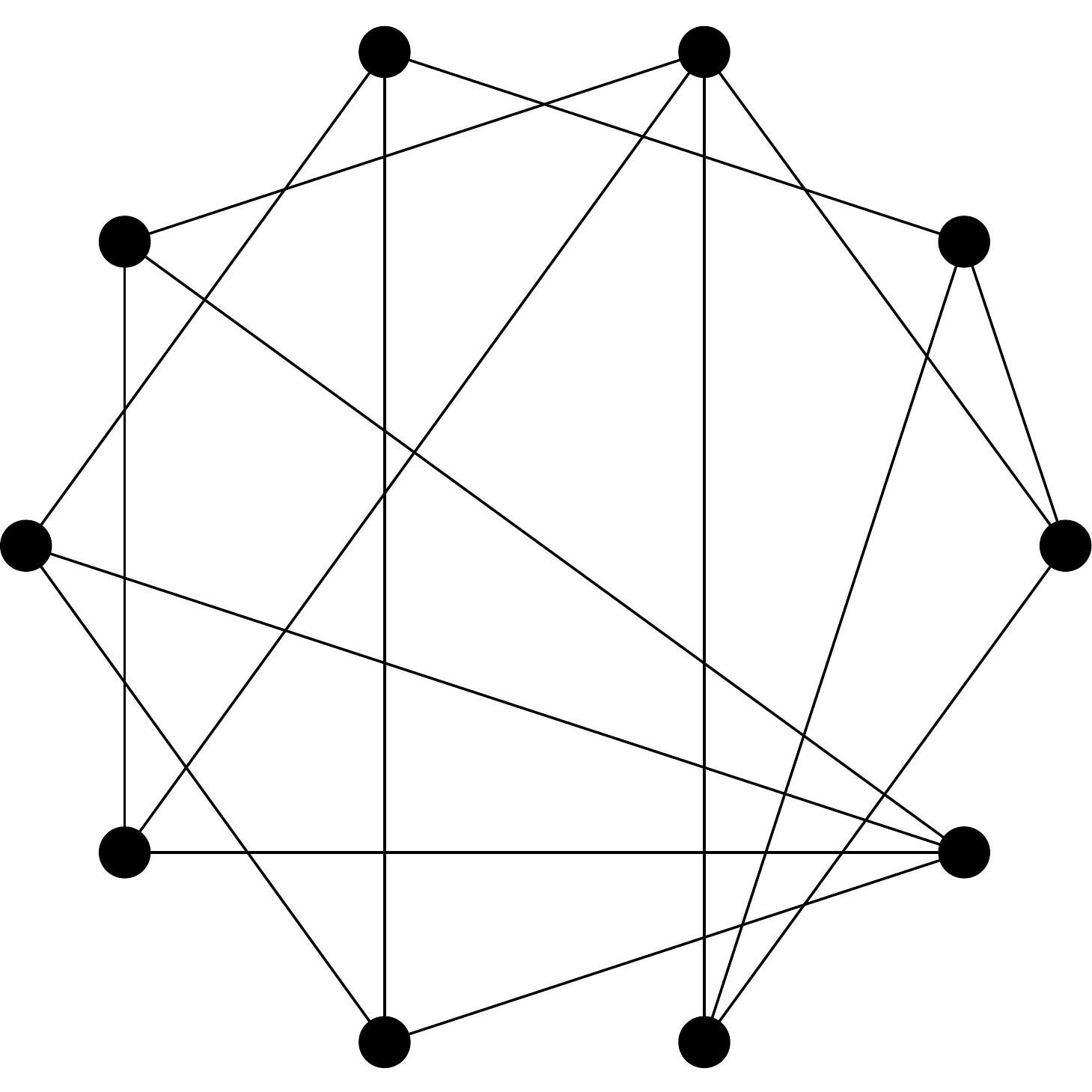}\ \ 
\includegraphics[width=.15\textwidth]{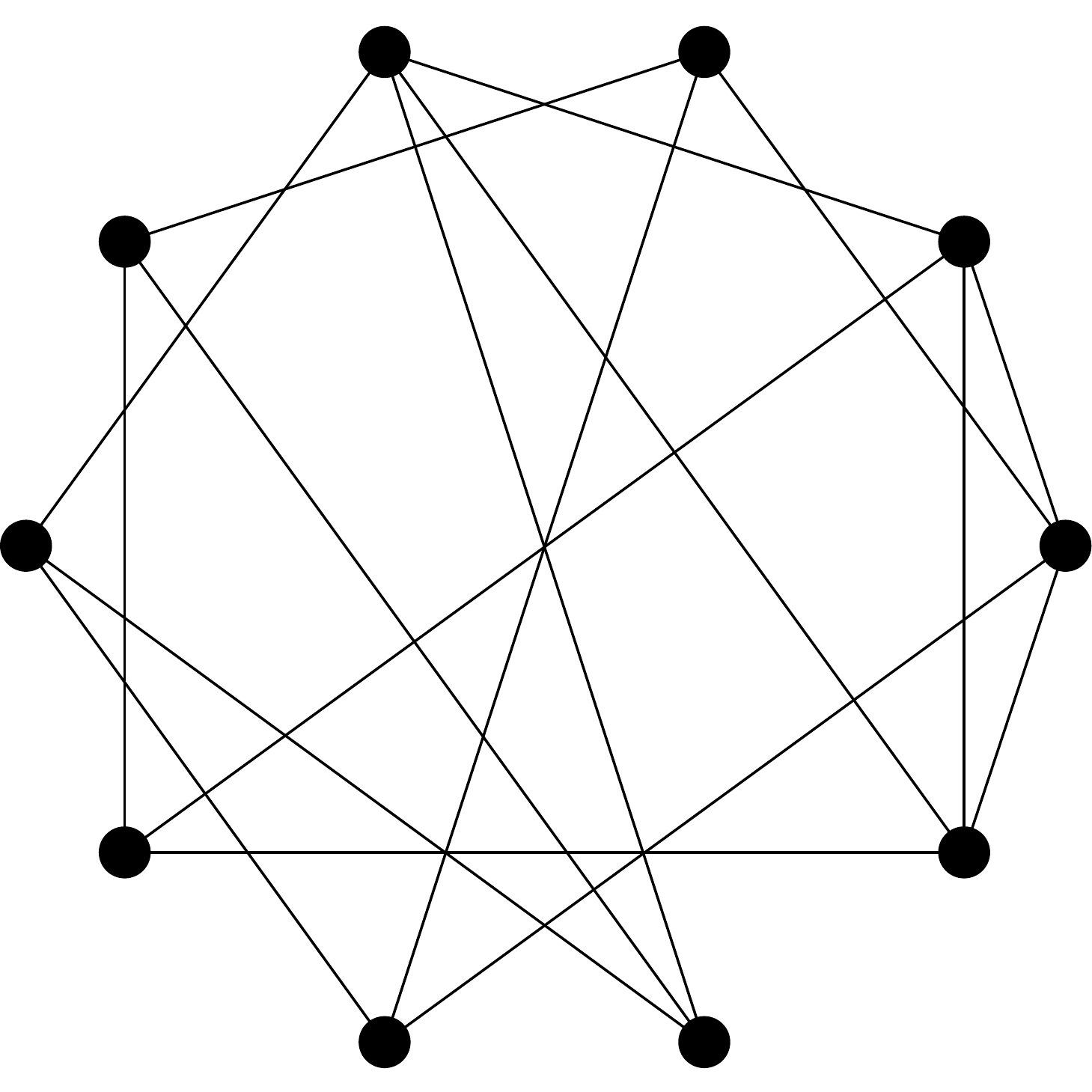}\ \ 
\includegraphics[width=.15\textwidth]{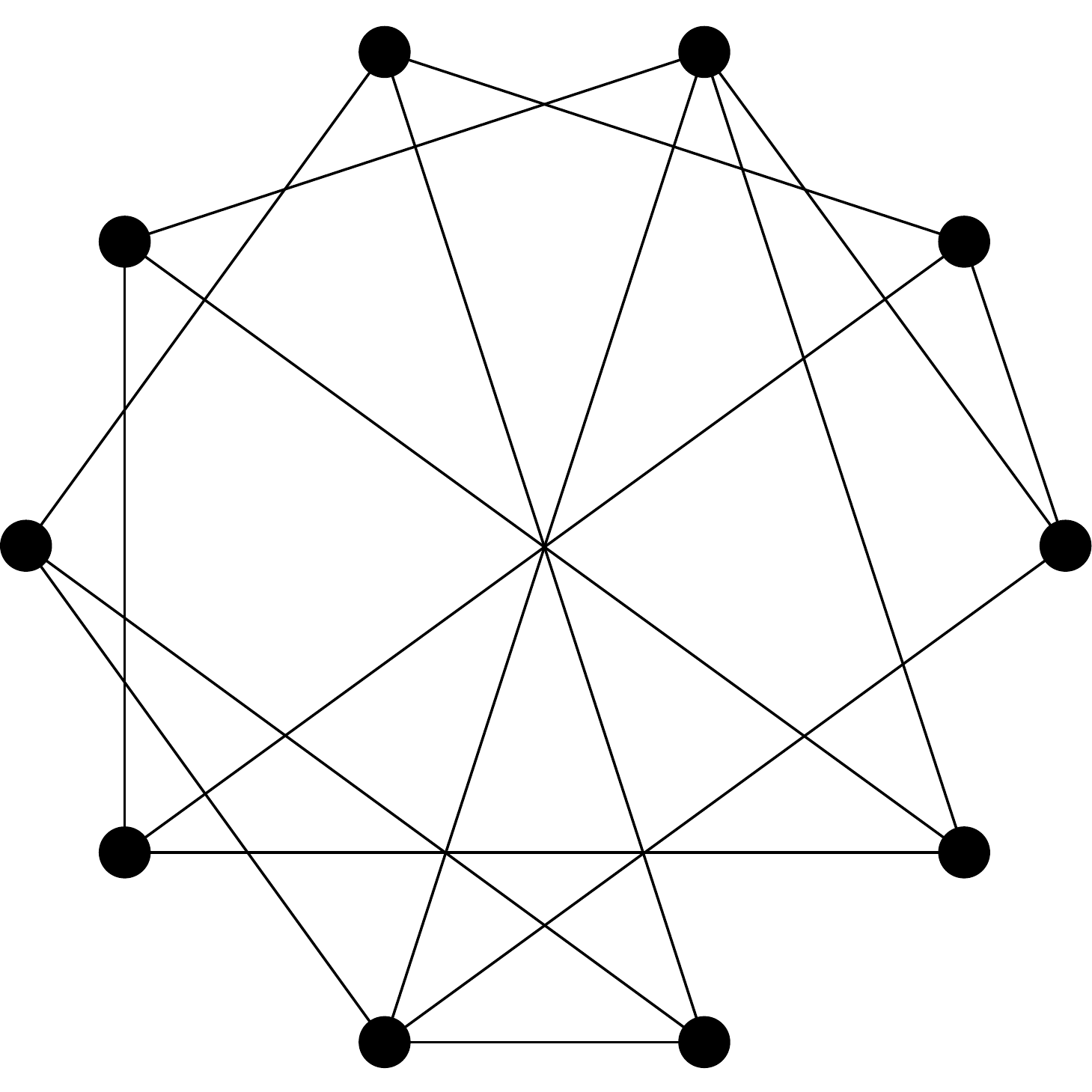}\ \ 
\includegraphics[width=.15\textwidth]{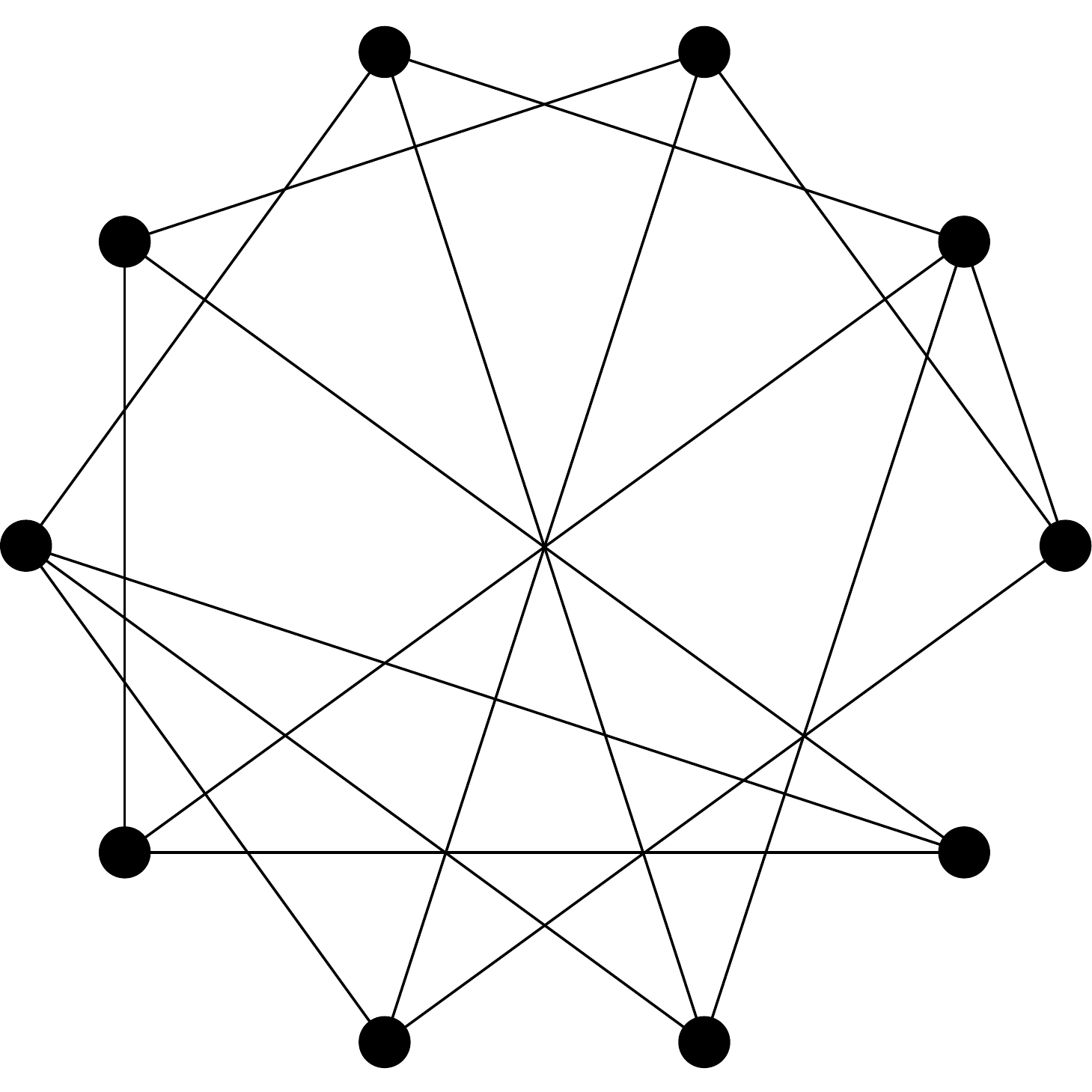}

\bigskip

\includegraphics[width=.15\textwidth]{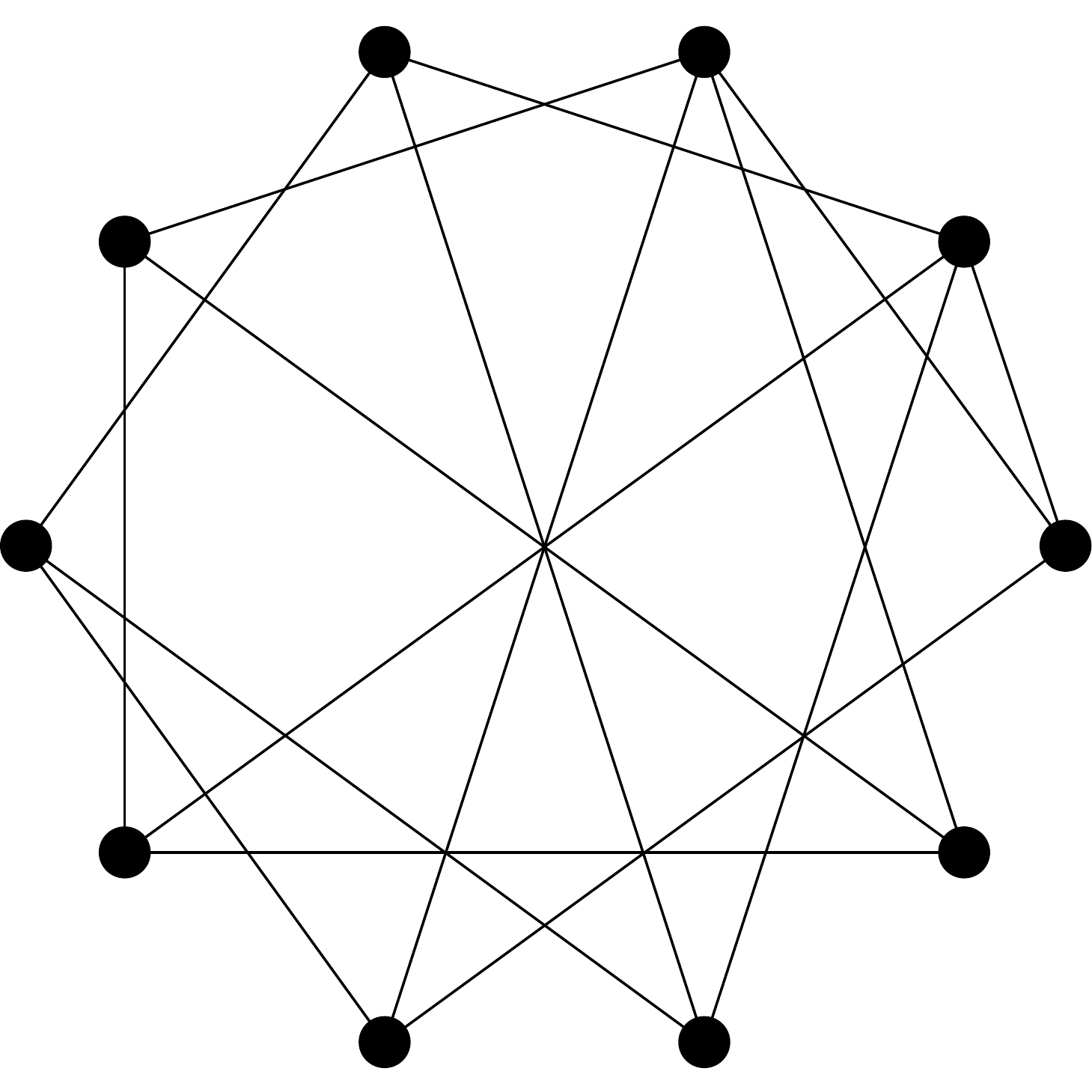}\ \ 
\includegraphics[width=.15\textwidth]{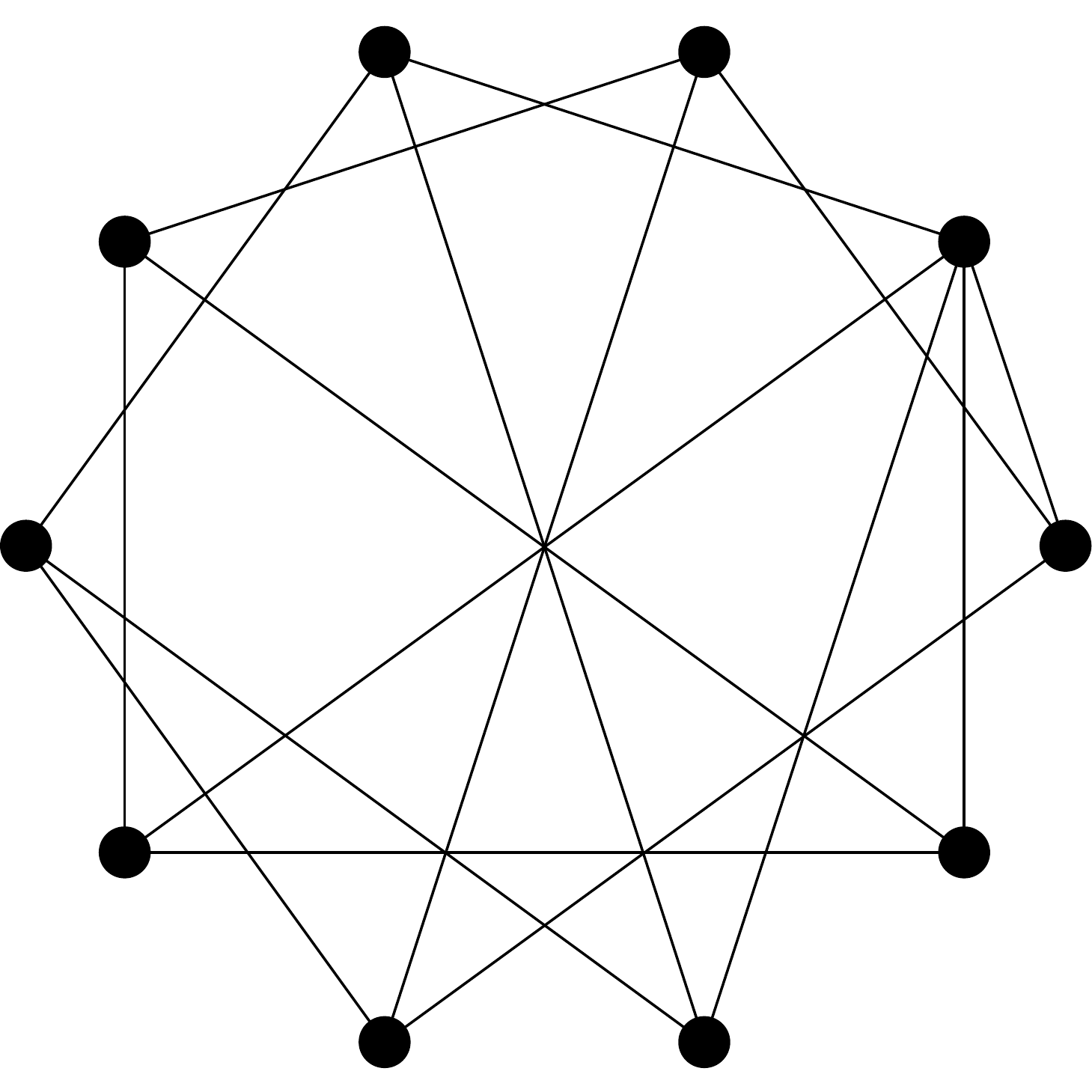}\ \ 
\includegraphics[width=.15\textwidth]{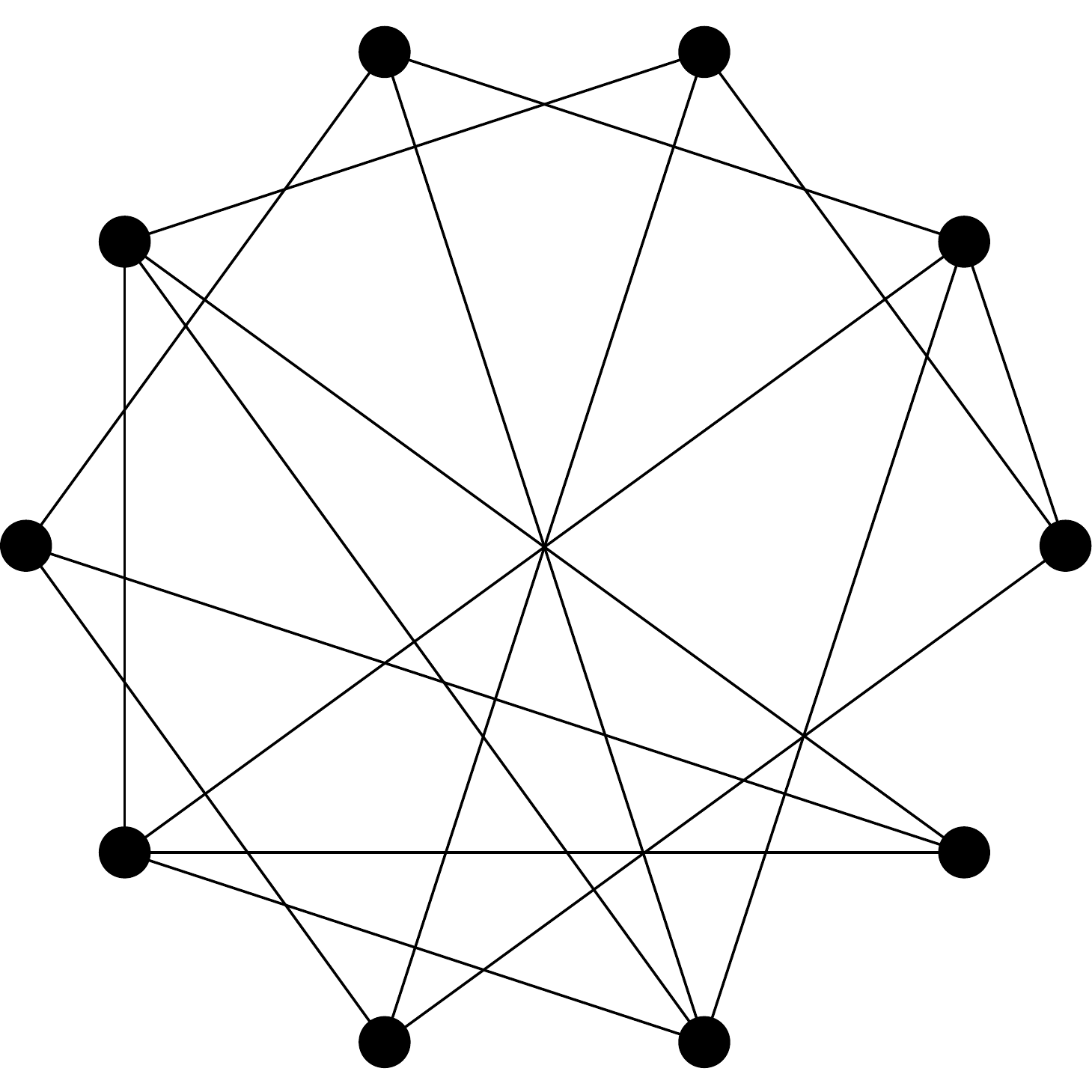}\ \ 
\includegraphics[width=.15\textwidth]{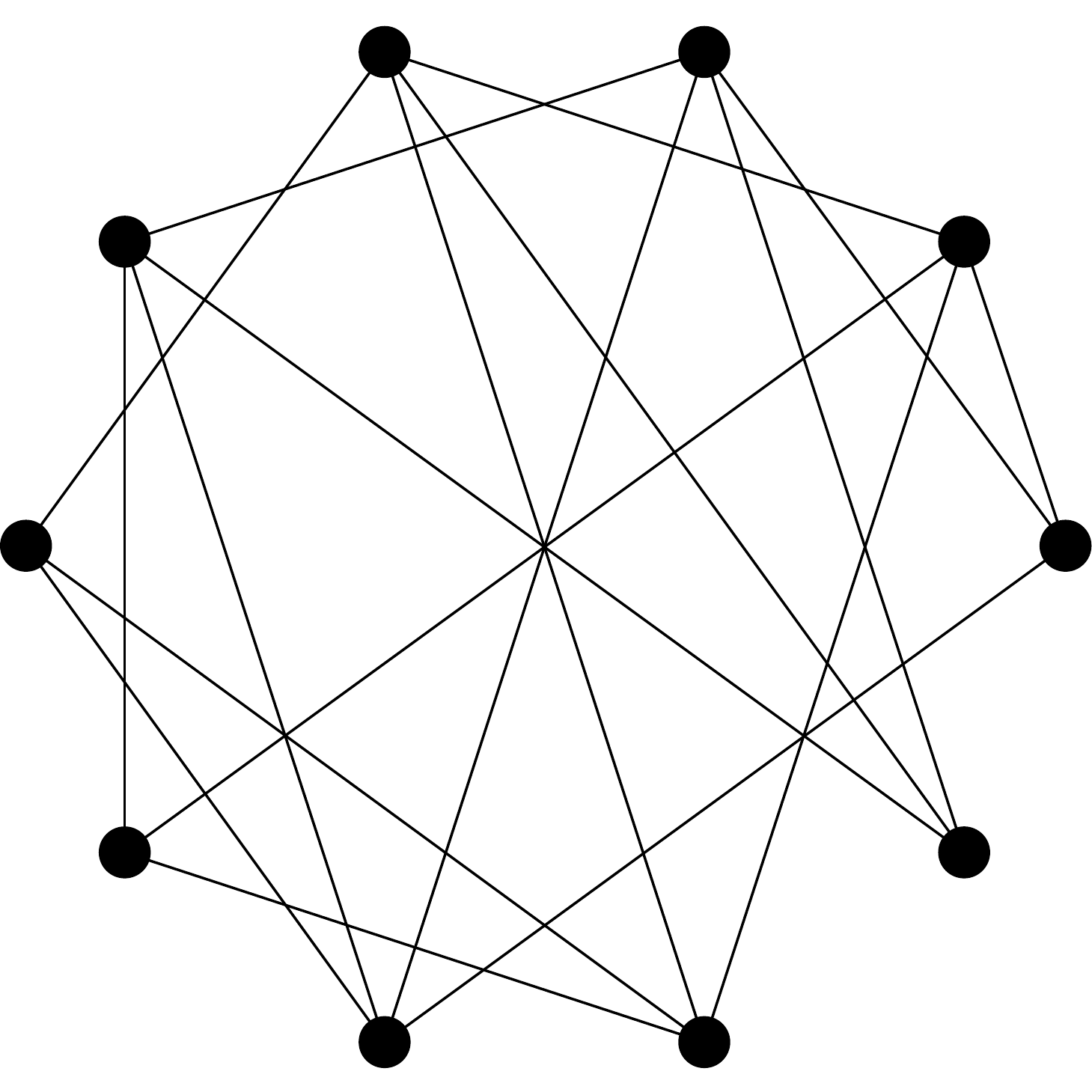}\ \ 
\includegraphics[width=.15\textwidth]{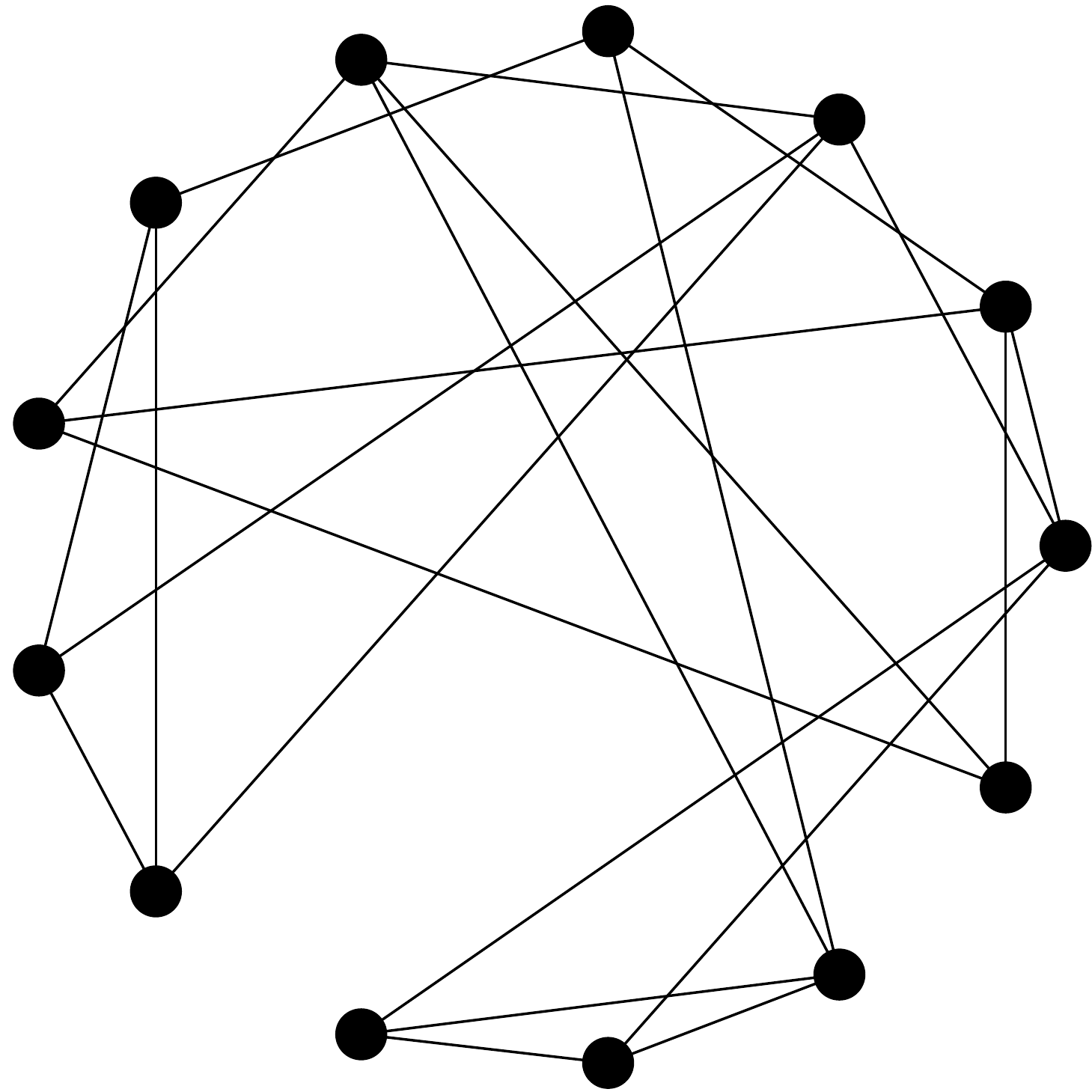}

\caption{All 17 4-critical $(P_7,C_4)$-free graphs.}
\label{fig:animals_N3P7C4}

\end{figure}

\begin{figure}[h!t]
\centering
\includegraphics[width=.15\textwidth]{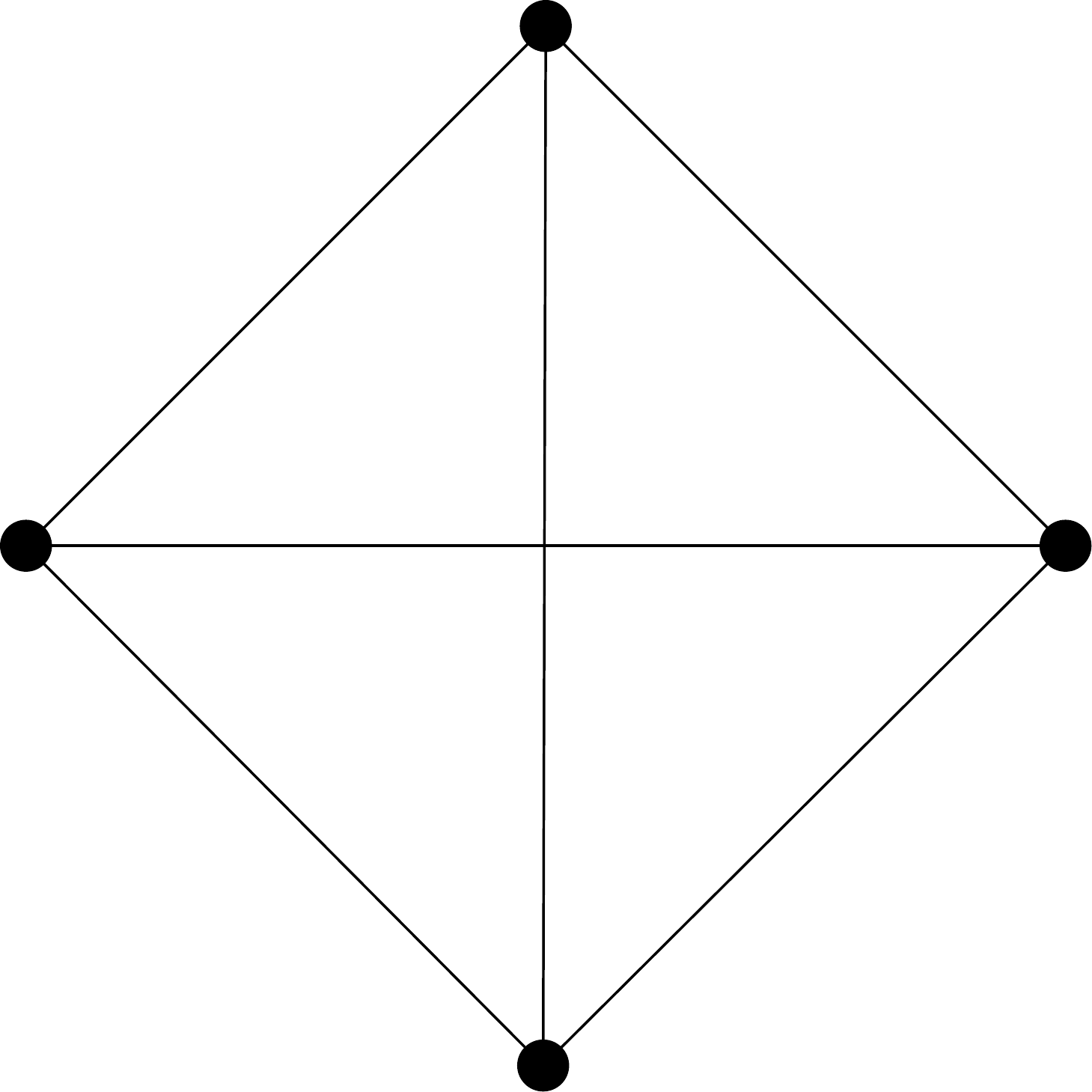}\ \ 
\includegraphics[width=.15\textwidth]{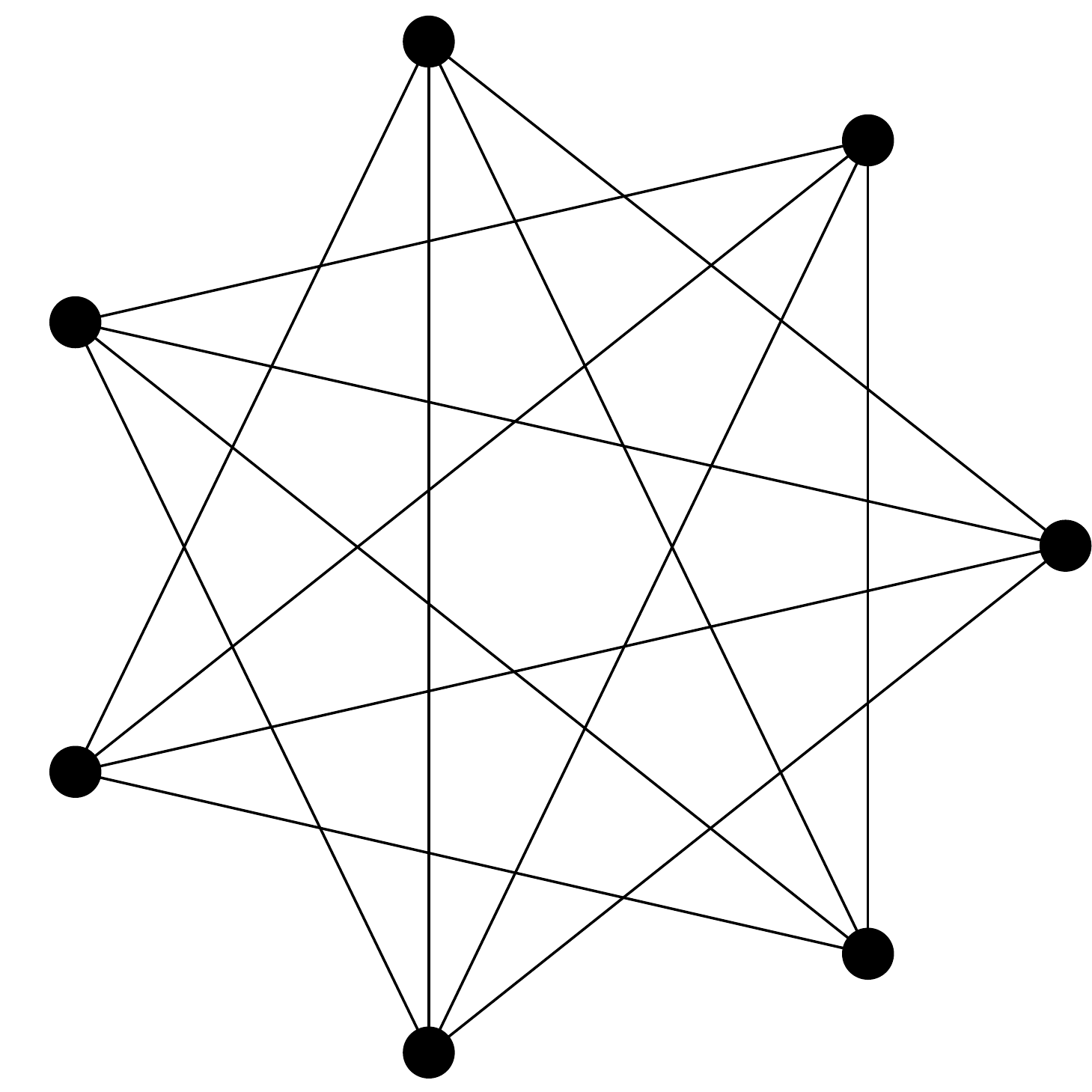}\ \ 
\includegraphics[width=.15\textwidth]{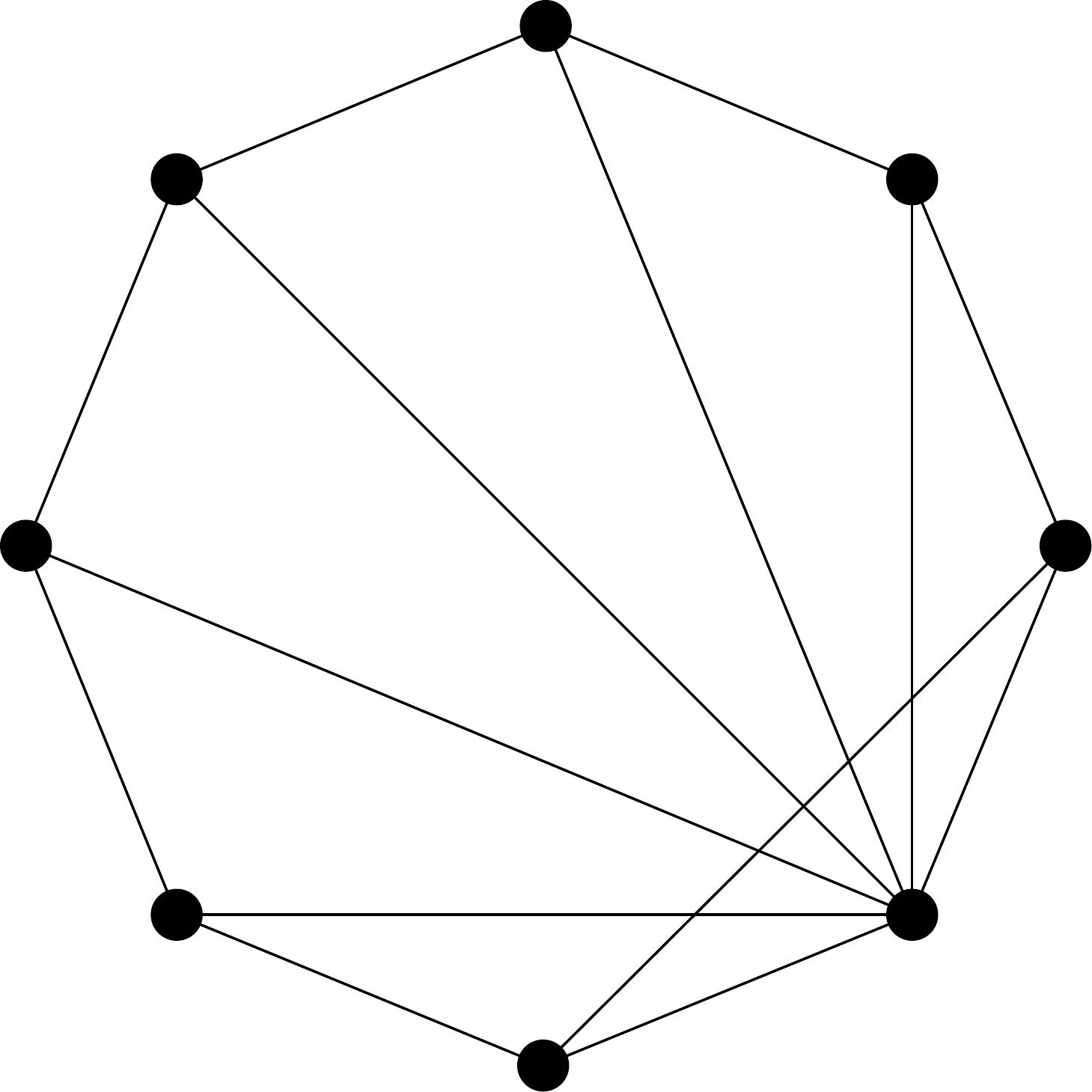}\ \ 
\includegraphics[width=.15\textwidth]{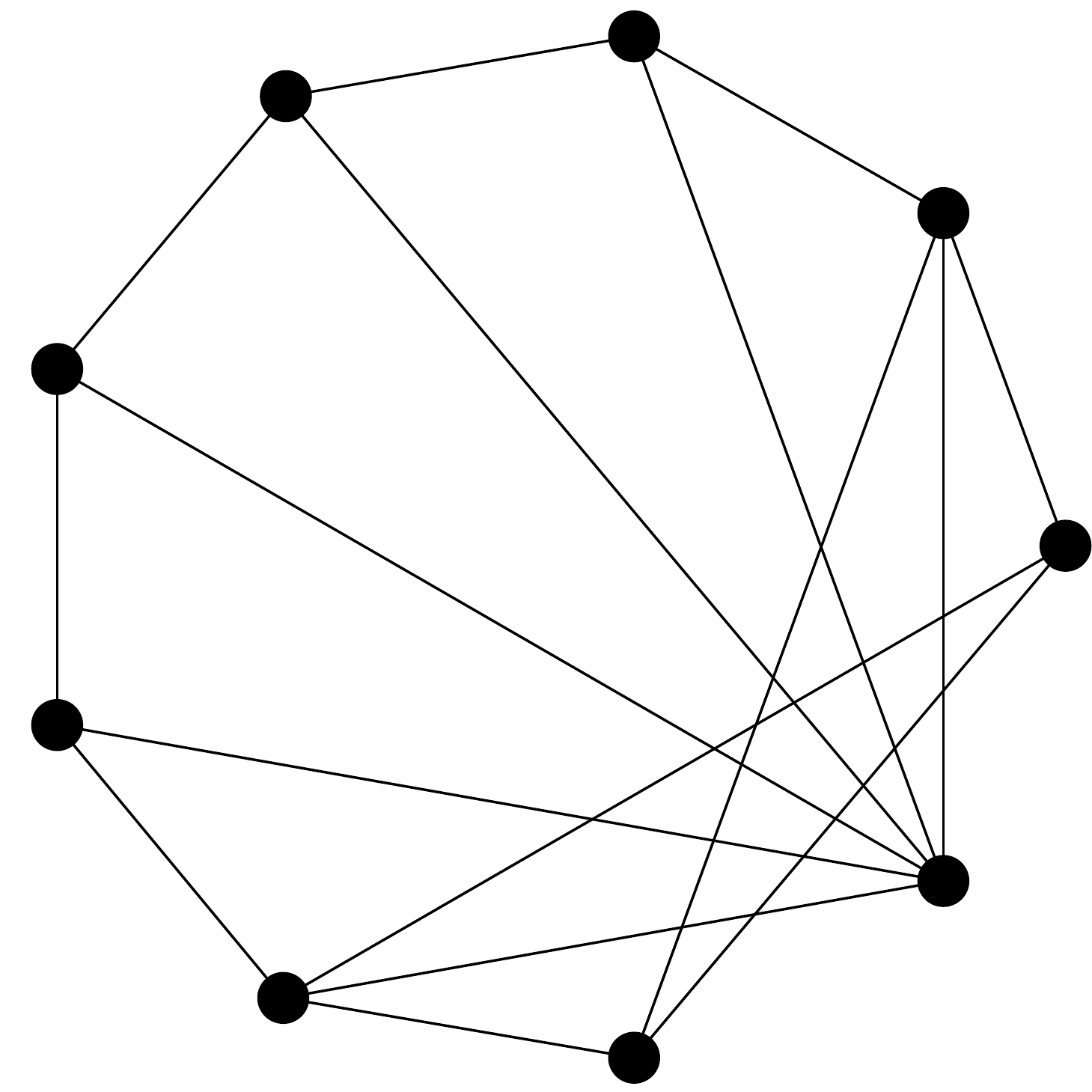}\ \ 
\includegraphics[width=.15\textwidth]{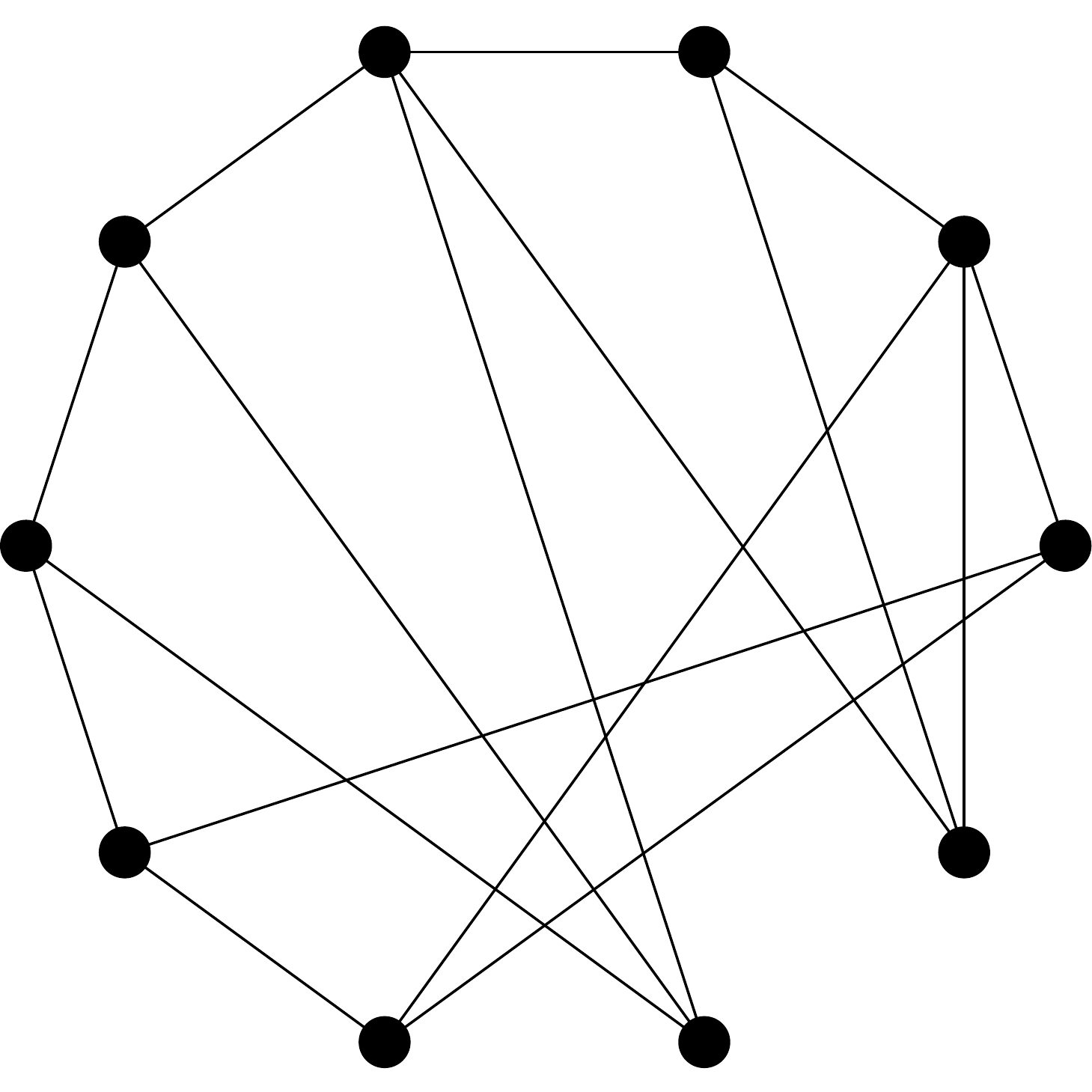}\ \ 
\includegraphics[width=.15\textwidth]{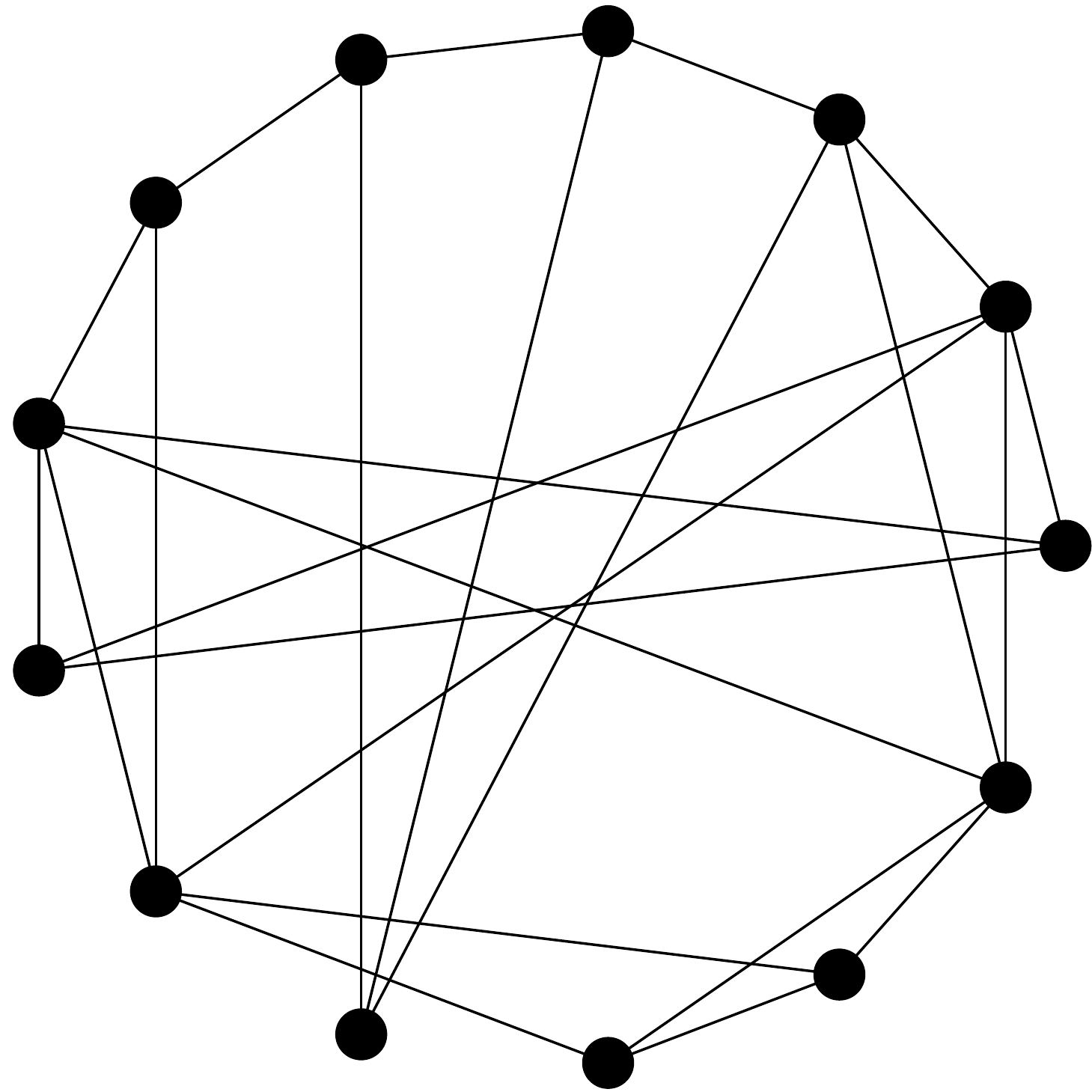}

\caption{All 6 4-critical $(P_7,C_5)$-free graphs.}
\label{fig:animals_N3P7C5}

\end{figure}

%

\begin{table}[ht!]
\centering
\begin{tabular}{| l | c c c c c c c | c |}
\hline 
Vertices 								& 4		& 6		&	7		&	8		&	9		&	10	& 13 	& total		\\ \hline
Critical graphs 				& 1		& 1		& 1		& 2		& 3		& 8		& 1		& 17			\\
Vertex-critical graphs 	& 1		& 1		& 1		& 2		& 4		& 24	& 2		& 35			\\
\hline 
\end{tabular}
\caption{Counts of all 4-critical and 4-vertex-critical $(P_7,C_4)$-free graphs.}

\label{table:counts_animals_N3P7C4}
\end{table}


%

\begin{table}[ht!]
\centering
\begin{tabular}{| l | c c c c c c c c c c | c |}
\hline 
Vertices 								& 4		& 6		&	7		&	8		&	9		&	10		& 11		& 12 		&	13		&	14		& total		\\ \hline
Critical graphs 				& 1		& 1		& 1		& 2		& 3		& 15		& 28		& 34		& 8			& 1			& 94			\\
Vertex-critical graphs 	& 1		& 1		& 1		& 2		& 4		& 33		& 54		& 53		&	14		& 1			& 164			\\
\hline 
\end{tabular}
\caption{Counts of all 4-critical and 4-vertex-critical $(P_8,C_4)$-free graphs.}

\label{table:counts_animals_N3P8C4}
\end{table}


%

\begin{table}[ht!]
\centering
\begin{tabular}{| l | c c c c c c | c |}
\hline 
Vertices 								& 4		& 7		&	8		&	9		&	10	&	13	& total		\\ \hline
Critical graphs 				& 1		& 1		& 1		& 1		& 1		& 1		& 6				\\
Vertex-critical graphs 	& 1		& 1		& 1		& 6		& 17	& 1		& 27			\\
\hline 
\end{tabular}
\caption{Counts of all 4-critical and 4-vertex-critical $(P_7,C_5)$-free graphs.}

\label{table:counts_animals_N3P7C5}
\end{table}

We did not succeed in solving the $(P_7,C_3)$-free case. However, we were able to determine all 4-critical $(P_7,C_3)$-free graphs with at most 35 vertices. These are shown in Figure~\ref{fig:animals_N3P7C3}, the counts are shown in Table~\ref{table:counts_animals_N3P7C3} and their adjacency lists can be found in Appendix~2. Since the largest 4-critical $(P_7,C_3)$-free graph up to 35 vertices has only 16 vertices, we conjecture the following.

\begin{conjecture}\label{conj:N3P7C3}
The seven graphs from Figure~\ref{fig:animals_N3P7C3} are the only 4-critical $(P_7,C_3)$-free graphs.
\end{conjecture}

\begin{figure}[h!t]
\centering
\includegraphics[width=.20\textwidth]{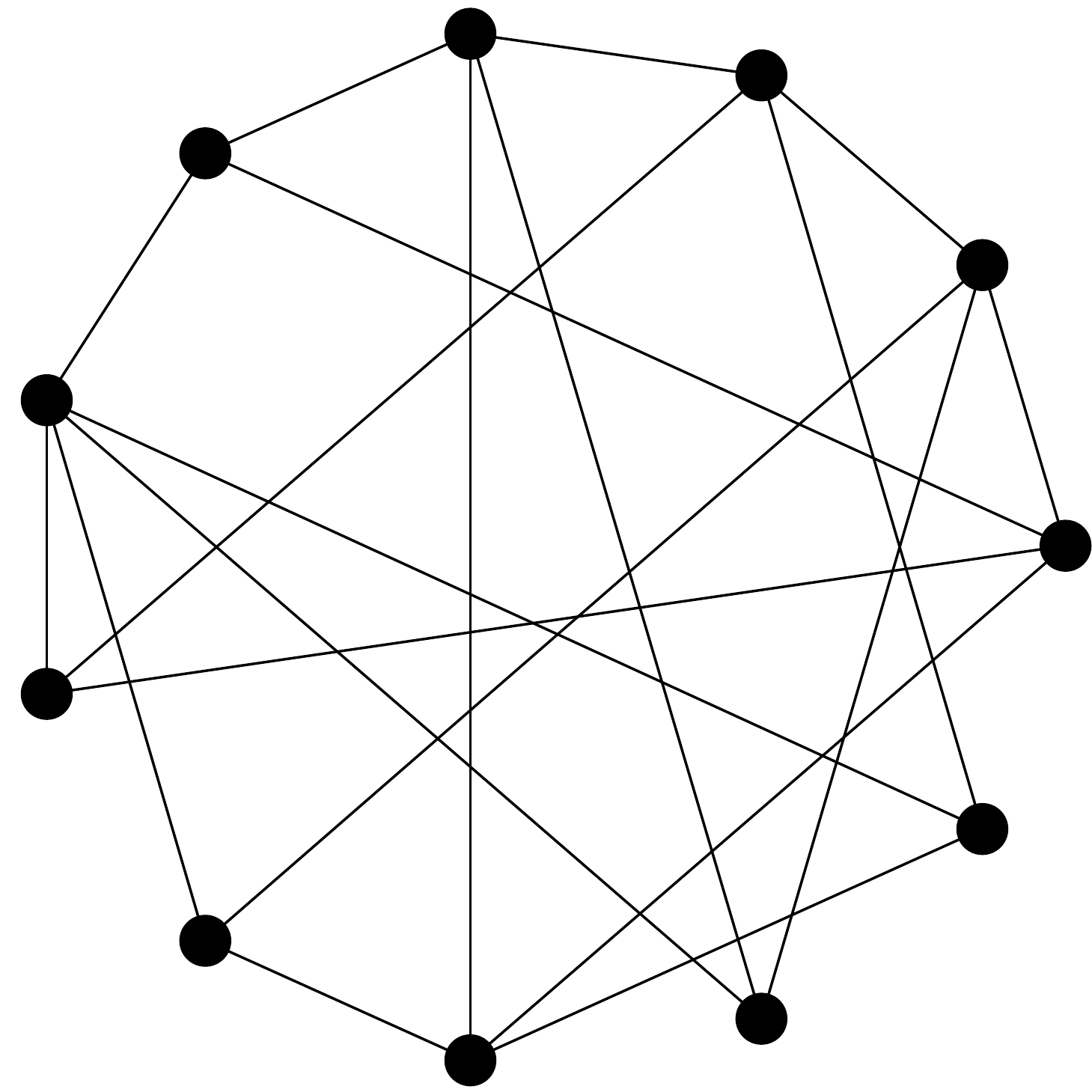}\ \ 
\includegraphics[width=.20\textwidth]{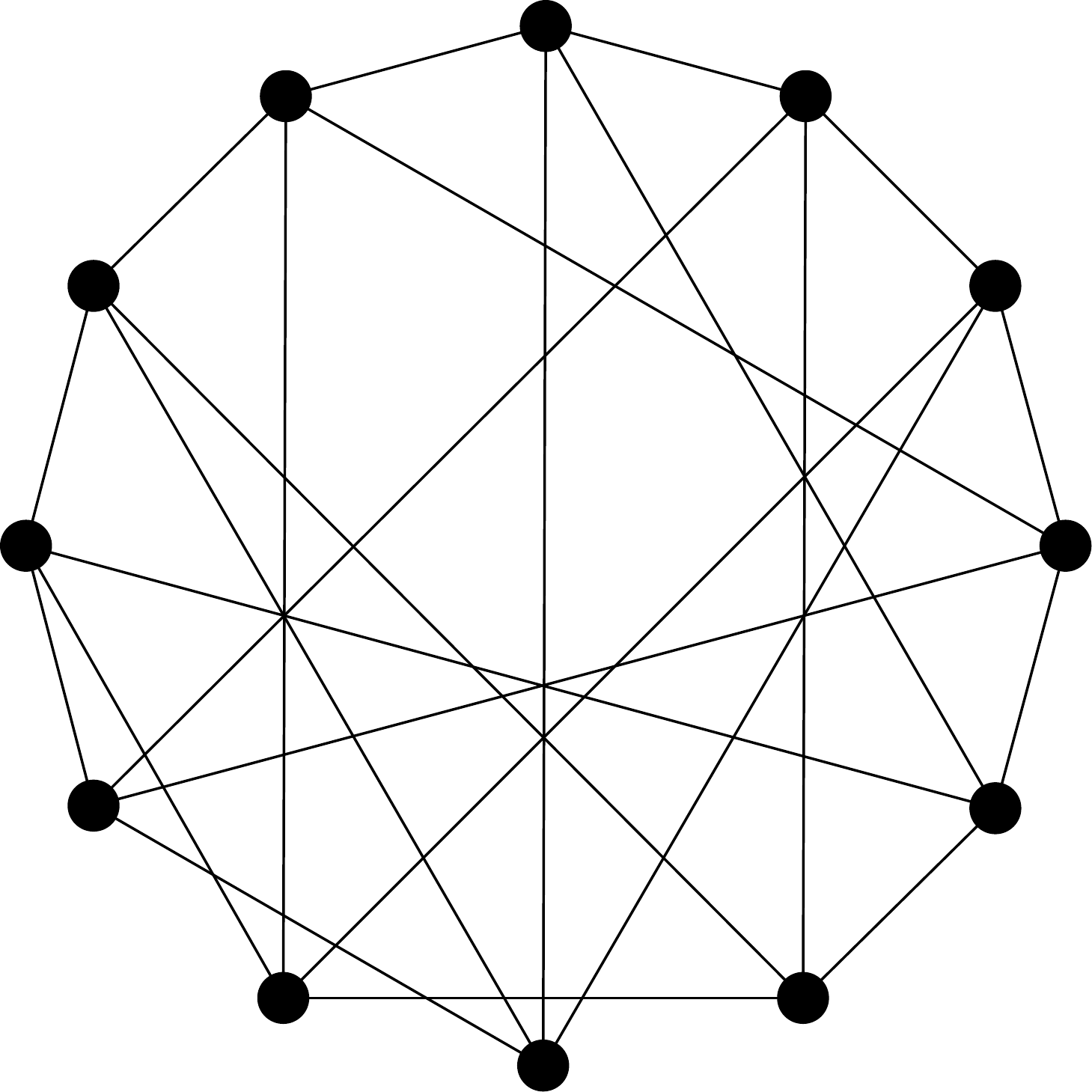}\ \ 
\includegraphics[width=.20\textwidth]{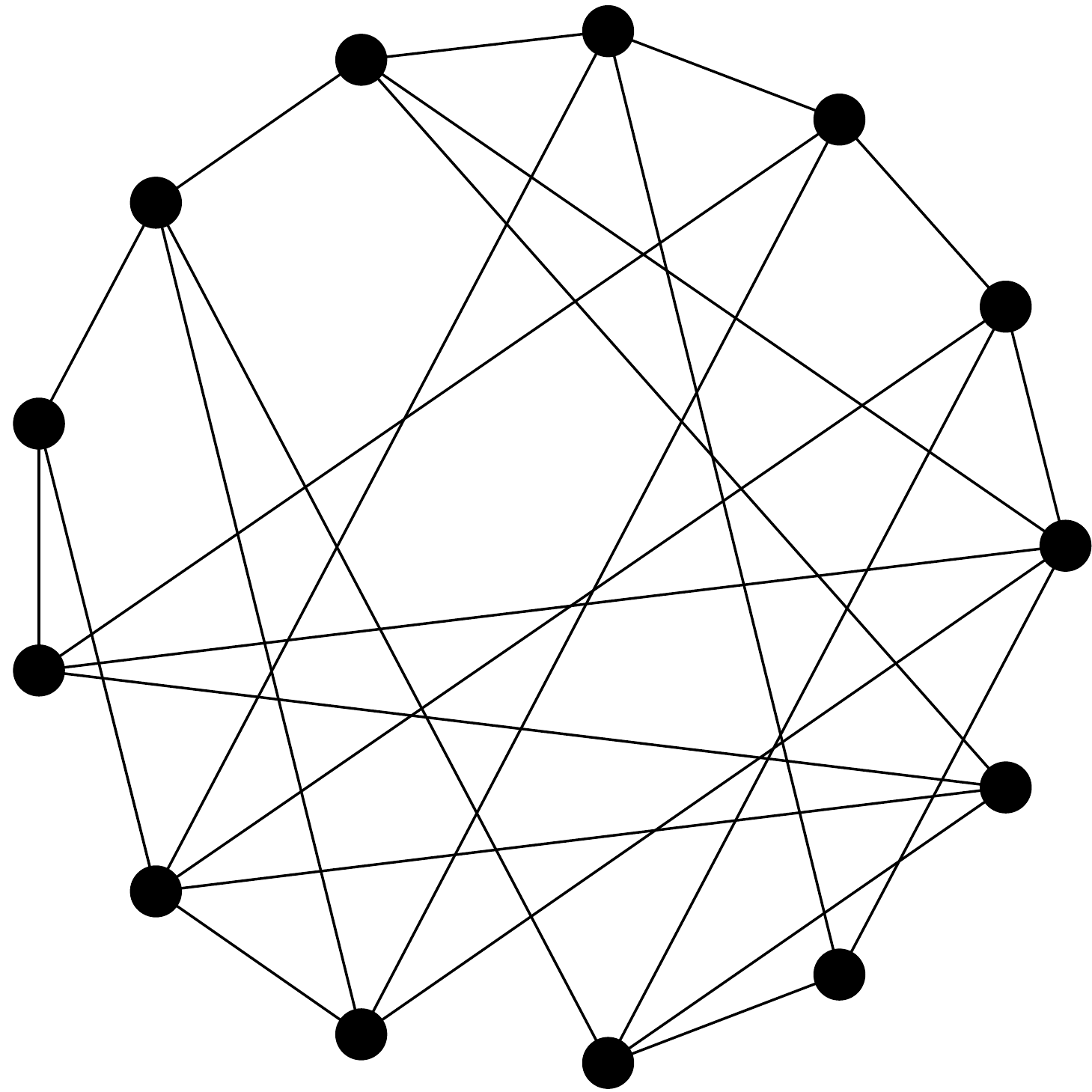}\ \ 
\includegraphics[width=.20\textwidth]{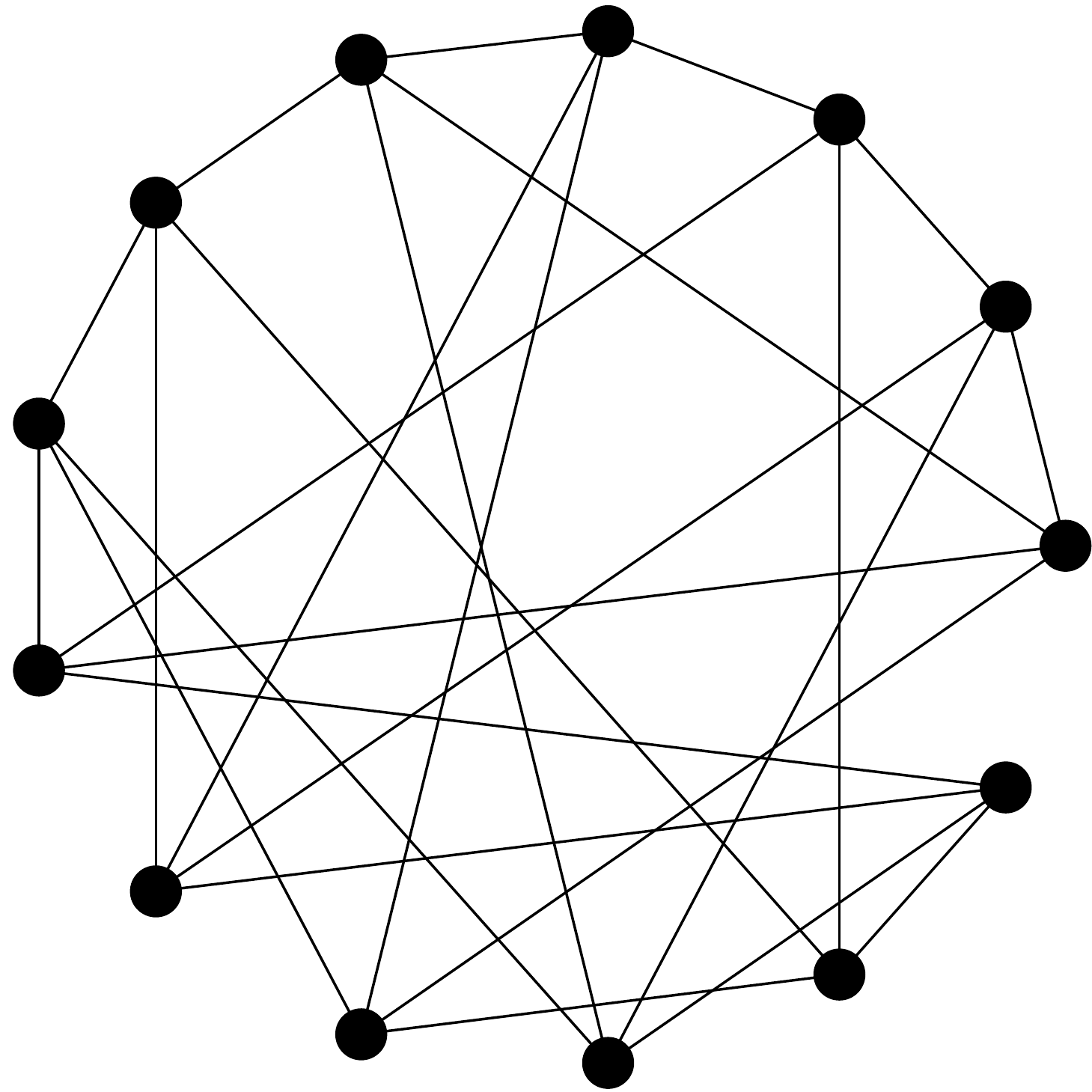}

\bigskip
 
\includegraphics[width=.20\textwidth]{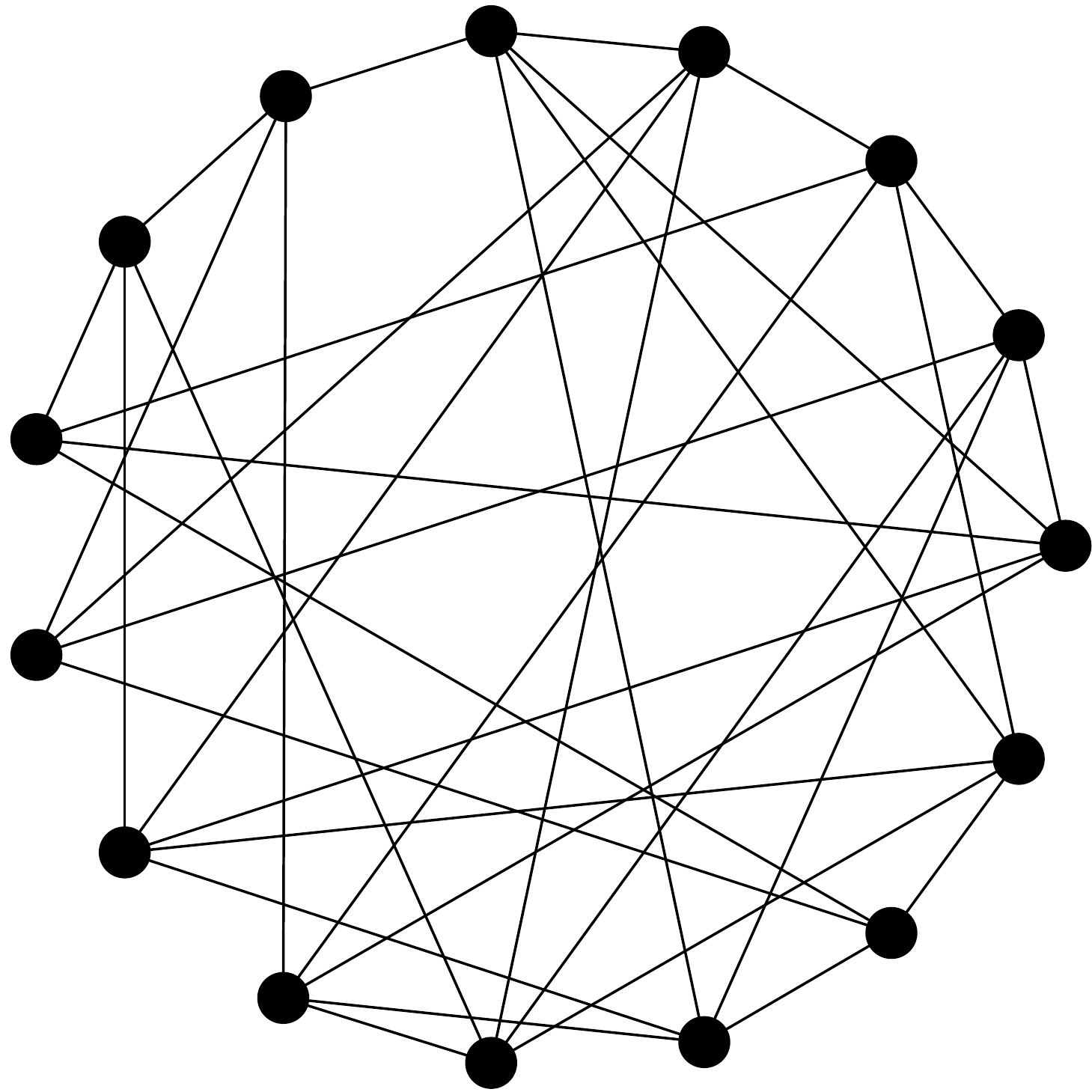}\ \ 
\includegraphics[width=.20\textwidth]{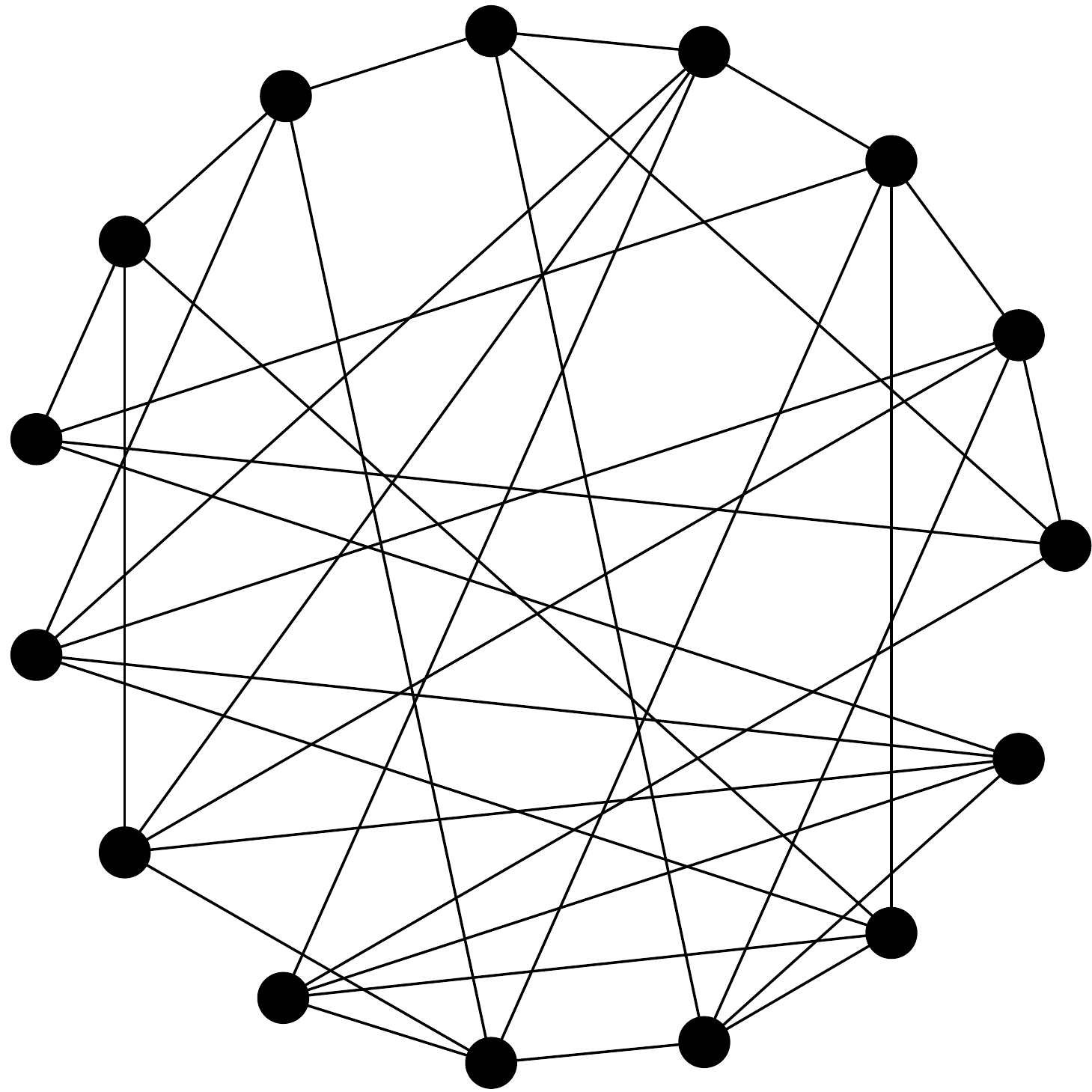}\ \ 
\includegraphics[width=.20\textwidth]{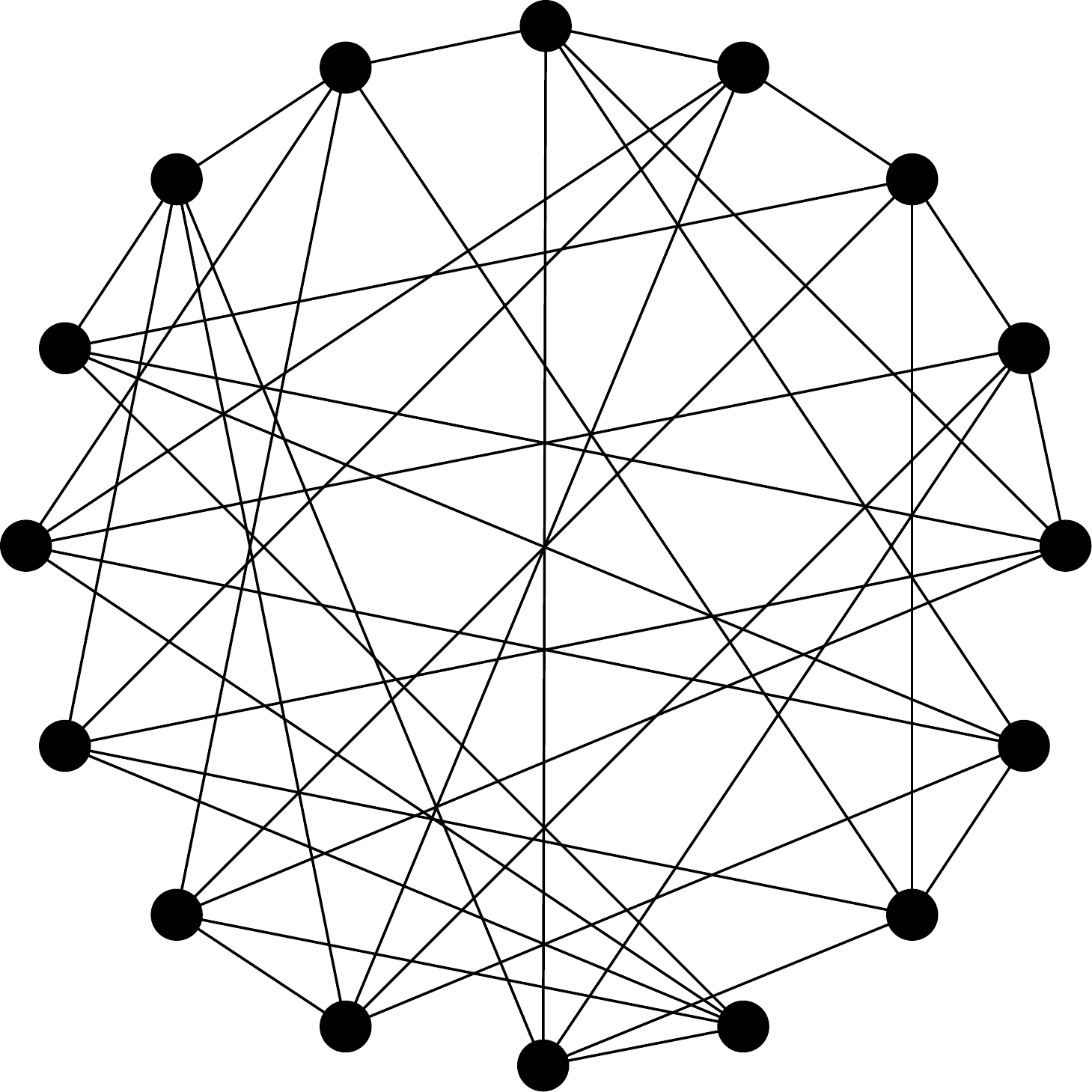}

\caption{All seven 4-critical $(P_7,C_3)$-free graphs with at most 35 vertices.}
\label{fig:animals_N3P7C3}

\end{figure}

%

\begin{table}[ht!]
\centering
\begin{tabular}{| l | c c c c c | c |}
\hline 
Vertices 								& 11		& 12		&	13		&	15		&	16	& total		\\ \hline
Critical graphs 				& 1			& 1			& 2			& 2			& 1		& 7				\\
\hline 
\end{tabular}
\caption{Counts of all 7 4-critical $(P_7,C_3)$-free graphs with at most 35 vertices. The number of 4-vertex-critical graphs is the same.}

\label{table:counts_animals_N3P7C3}
\end{table}

\subsection{4-critical graphs with a given minimal girth}
\label{sect:girth}

Golovach et al.~\cite{golovach2014coloring} mention the Brinkmann graph, constructed by Brinkmann and Meringer in~\cite{brinkmann1997smallest},
as an example of a $P_{10}$-free of girth five which is not 3-colorable. 
However, this claim is not entirely correct since the Brinkmann graph contains a $P_{12}$ as an induced subgraph. 
Using Algorithm~\ref{algo:init-algo} we can in fact show that every $P_{11}$-free graph with girth at least five is 3-colorable.
We also improve upon results of Golovach et al.~\cite{golovach2014coloring} in the cases of girth at least six, and at least seven. These results are also summarized in Table~\ref{table:bounds_girth}.

\begin{theorem}\label{thm:girth}
The following assertions hold.
\begin{enumerate}[(a)]
	\item Every $P_{11}$-free graph of girth at least five is 3-colorable. 
	\item There is a 4-chromatic $P_{12}$-free graph of girth five, and the smallest such graph has 21 vertices.
	\item Every $P_{14}$-free graph of girth at least six is 3-colorable.
	\item Every $P_{17}$-free graph of girth at least seven is 3-colorable.
\end{enumerate}
\end{theorem}
\begin{proof}
Algorithm~\ref{algo:init-algo} terminates in the cases of (a), (c), and (d). 
It outputs an empty list of 4-critical graphs.
Thus, by Theorem~\ref{lem:diamond-algo-correct}, there are no 4-critical graphs in the respective graph classes, which means that all graphs therein are 3-colorable.

To prove (b), we modified Algorithm~\ref{algo:init-algo} such that it discards any graph with more than 30 vertices. Indeed no other 4-critical graphs than one graph on 21 vertices were found.
By Theorem~\ref{lem:diamond-algo-correct}, (b) follows.
\end{proof} 

We remark that the 4-critical $P_{12}$-free graph with girth five mentioned in Theorem~\ref{thm:girth}.(b) is the only such graph up to at least 30 vertices. 
This graph is shown in Figure~\ref{fig:N3P12g5} and its adjacency list can be found in Appendix~2. 
It can also be obtained from the \textit{House of Graphs}~\cite{hog} by searching for the keywords ``4-critical P12-free * girth 5''. 



Note that a graph with girth at least five cannot contain similar vertices $(u,v)$ which both have degree at least two. Experiments showed that when searching for critical $P_t$-free graphs with a given minimal girth, it is best only to try to apply the following expansion rules in Algorithm~\ref{algo:construct}: poor vertices (line~\ref{line:smallvertex}), weak cycles (line~\ref{line:weakcycle}) and cutvertices (line~\ref{line:clique}). This saves a significant amount of CPU time since then one does not have to search for similar \textit{elements} or compute \textit{$k$-hulls}.

\begin{table}[ht!]
\centering
\begin{tabular}{| c | l | l |}
\hline 
Girth & Old lower bound~\cite{golovach2014coloring} & New lower bound\\
\hline
4 & $P_5$-free (exact) & $P_5$-free (exact)\\
5 & $P_7$-free & $P_{11}$-free (exact)\\
6 & $P_{10}$-free & $P_{14}$-free\\
7 & $P_{12}$-free & $P_{17}$-free\\
\hline
\end{tabular}
\caption{Old and new lower bounds such that every $P_k$-free graph with girth at least $g$ is 3-colorable.}
\label{table:bounds_girth}
\end{table}

\begin{figure}[h!t]
\centering
\includegraphics[width=.4\textwidth]{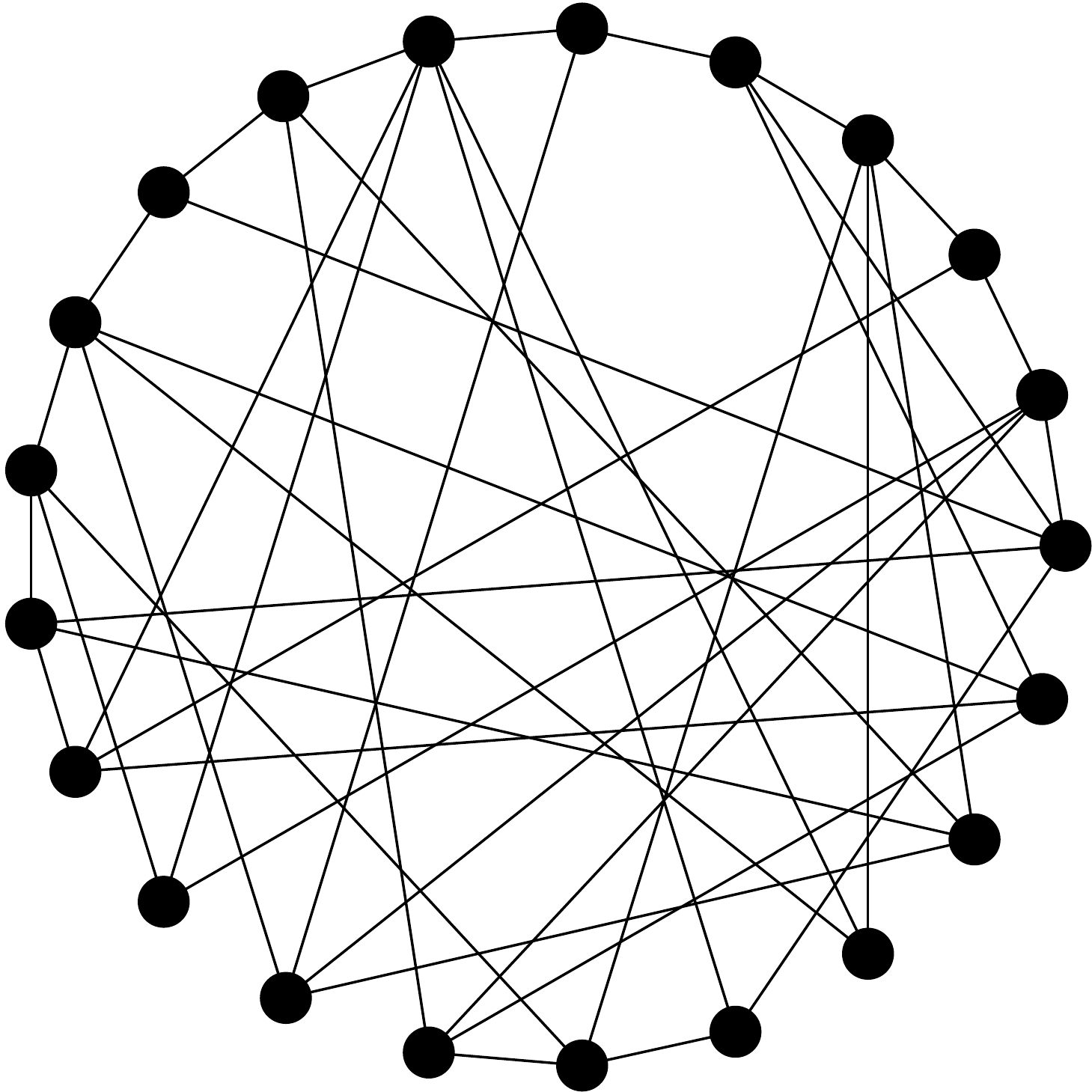}

\caption{The smallest 4-critical $P_{12}$-free graph with girth 5. It has 21 vertices and it is the only 4-critical $P_{12}$-free graph with girth at least five up to at least 30 vertices.}
\label{fig:N3P12g5}

\end{figure}

\subsection{4-critical planar graphs}

In this section we turn to planar critical graphs.
Due to the four-color theorem, we may restrict our attention to 4-critical graphs.
Let us first note that if an induced path is forbidden, there are only finitely many vertex-critical planar graphs.
This implies that there are only finitely many 4-critical planar $P_t$-free graphs, for all $t \in \mathbb N$.

\begin{theorem}\label{thm:planarfinite}
For any integer $t$ there are only finitely many vertex-critical graphs that are both planar and $P_t$-free.
\end{theorem}

For the proof of Theorem~\ref{thm:planarfinite}, we need the following result of B\"ohme et al.~\cite{BMSS04}, which we state in a slightly different fashion than in the original paper.

\begin{theorem}[B\"ohme et al.~\cite{BMSS04}]\label{thm:PtorK2s}
For every $k,s,t \in  \mathbb N$ there is a number $n=n(k,s,t)$ such that every $k$-connected graph on at least $n$ vertices contains an induced path on $t$ vertices or a subdivision of $K_{2,s}$.
\end{theorem}

From this result our theorem follows readily.

\begin{proof}[Proof of Theorem~\ref{thm:planarfinite}]
Let $n=n(2,2t,t)$ be the number promised by Theorem~\ref{thm:PtorK2s}.
Suppose for a contradiction that there is a planar $P_t$-free vertex-critical graph $G=(V,E)$ on at least $n$ vertices.
Clearly $G$ is connected.
We consider $G$ to be embedded into the plane.

By Theorem~\ref{thm:PtorK2s} and since $G$ is $P_t$-free, $G$ contains a subdivision $H$ of $K_{2,2t}$ as a subgraph.
We consider $H$ as an embedded subgraph of $G$.
Let $u$ and $v$ be the two vertices of degree $2t$ in $H$, and let $P^1, \ldots , P^{2t}$ be the mutually vertex disjoint paths from $u$ to $v$ in $H$.
W.l.o.g.~$P^i \cup P^{i+1}$ is the border of an inner face of $H$, for each $i=1,\ldots,2t-1$, and $P^1 \cup P^{2t}$ is the border of the outer face.

First we assume that $G-\{u,v\}$ is connected.
We pick two vertices $x \in V(P^1) \setminus \{u,v\}$ and $y \in V(P^t) \setminus \{u,v\}$.
Let $Q$ be a shortest path from $x$ to $y$ in $G-\{u,v\}$.
Due to the planarity of $G$, necessarily $Q$ contains a vertex of each $P^i$, $i = 1, \ldots, t$, or of each $P^j$, $j = 1,t,t+1, \ldots, 2t$.
So, $Q$ is an induced path on at least $t$ vertices, in contradiction to the fact that $G$ is $P_t$-free.

So we know that $\{u,v\}$ is a cutset of $G$.
Thus $G-\{u,v\}$ has exactly two components, say $G_1$ and $G_2$, a folklore fact about vertex-critical graphs.
We may assume that the paths $P^1-\{u,v\},\ldots,P^{t}-\{u,v\}$ are contained in $G_1$.
We again pick vertices $x \in V(P^1) \setminus \{u,v\}$ and $y \in V(P^t) \setminus \{u,v\}$.
Let $Q$ be a shortest path from $x$ to $y$ in $G_1$.
Since $G_2$ is not empty and $G$ is planar, $Q$ must contain a vertex of each $P^i$, $i = 1, \ldots, t$.
But this is again a contradiction to the fact that $G$ is $P_t$-free. 
\end{proof}

We modified Algorithm~\ref{algo:construct} so it only searches for \textit{planar} $k$-critical $\mathcal{H}$-free graphs by adding a test for planarity on line~\ref{line:isocheck}. We used the algorithm of Boyer and Myrvold~\cite{boyer2004cutting} to test if a graph is planar. Using this modified version of Algorithm~\ref{algo:construct} we obtained the following result.

\begin{theorem}\label{thm:planarresults}
There are exactly 27 4-critical $(P_7,C_6)$-free planar graphs.
\end{theorem}

Note that, as mentioned earlier, if the planarity condition is dropped there are infinitely many 4-critical $(P_7,C_6)$-free graphs. The counts of the 4-critical and 4-vertex-critical planar $(P_7,C_6)$-free graphs can be found in Table~\ref{table:counts_animals_N3P7C6_planar}. The graphs themselves can be downloaded from \url{http://hog.grinvin.org/Critical}

%

\begin{table}[ht!]
\centering
\begin{tabular}{| l | c c c c c c c c | c |}
\hline 
Vertices 								& 4		& 6		&	7		&	8		&	9			&	10	& 11 	&	12	& total		\\ \hline
Critical graphs 				& 1		& 1		& 2		& 2		& 10		& 9		& 1		& 1		&	27	\\
Vertex-critical graphs 	& 1		& 1		& 6		& 2		& 52		& 118	& 2		& 3		&	185	\\
\hline 
\end{tabular}
\caption{Counts of all 4-critical and 4-vertex-critical planar $(P_7,C_6)$-free graphs.}

\label{table:counts_animals_N3P7C6_planar}
\end{table}

It follows from Gr\"otzsch's Theorem~\cite{grotzsch1958theorie} that there are no planar 4-critical $(P_7,C_3)$-free graphs.
The planar 4-critical $(P_7,C_4)$-free and $(P_7,C_5)$-free graphs can easily be obtained by testing which of the critical graphs from Theorem~\ref{thm:N3P7Ck} are planar. 

We did not succeed in solving the general planar $P_7$-free case, in the sense that we could not determine the exact list of 4-critical graphs in this class. However, we were able to determine all planar 4-critical $P_7$-free graphs with at most 30 vertices. Their counts are shown in Table~\ref{table:counts_animals_N3P7planar}. These graphs can also be obtained from the \textit{House of Graphs}~\cite{hog} by searching for the keywords ``planar 4-critical P7-free''.  Since the largest planar 4-critical $P_7$-free graph we found up to 30 vertices has 13 vertices, we conjecture the following.

\begin{conjecture}\label{conj:N3P7planar}
The 52 graphs from Table~\ref{table:counts_animals_N3P7planar} are the only planar 4-critical $P_7$-free graphs.
\end{conjecture}

%

\begin{table}[ht!]
\centering
\begin{tabular}{| l | c c c c c c c c c | c |}
\hline 
Vertices 								& 4		& 6		&	7		&	8		&	9			&	10	& 11 	&	12	& 13	& total		\\ \hline
Critical graphs 				& 1		& 1		& 2		& 2		& 14		& 19	& 4		& 6		&	3		&	52	\\
Vertex-critical graphs 	& 1		& 1		& 6		& 2		& 65		& 347	& 6		& 19	&	15	&	462	\\
\hline 
\end{tabular}
\caption{Counts of all planar 4-critical $P_7$-free graphs with at most 30 vertices.}

\label{table:counts_animals_N3P7planar}
\end{table}


\subsection*{Acknowledgements}

Several of the computations for this work were carried out using the Stevin Supercomputer Infrastructure at Ghent University.
Jan Goedgebeur is supported by a Postdoctoral Fellowship of the Research Foundation Flanders (FWO).


\bibliographystyle{amsplain}
\bibliography{bnm-j,bnm,references}


\section*{Appendix 1: Correctness testing}

The correctness of our algorithm is proven in Theorem~\ref{lem:diamond-algo-correct}, but it is also very important to thoroughly verify the correctness of our implementation to minimize the chance of programming errors. Therefore this section is dedicated to the description of the extensive correctness tests which we performed on our implementation.

Since all of our consistency tests passed, we believe that this is strong evidence for the correctness of our implementation.

More specifically, we performed the following consistency tests to verify the correctness of our generator for $k$-critical $\mathcal H$-free graphs (i.e.\ Algorithm~\ref{algo:init-algo}). The source code of this program can be downloaded from~\cite{criticalpfree-site}.

\begin{itemize}

\item As already mentioned in Section~\ref{sect:previous_results}, we verified that our program indeed yields the same results when we apply it to the following cases which were already known before in the literature (by hand or by computer):

\begin{itemize}
\item There are six 4-critical $P_5$-free graphs~\cite{BHS09}.
\item There are eight 5-critical $(P_5,C_5)$-free graphs~\cite{HMRSV15}.
\item The Gr\"otzsch graph is the only 4-critical $(P_6,C_3)$-free graph~\cite{randerath_04}.
\item There are four 4-critical $(P_6,C_4)$-free graphs~\cite{hell_14}.
\end{itemize}

\item We developed a completely independent generator for $k$-critical $P_t$-free graphs. Here we started from a generator for all graphs (i.e.\ the program \verb|geng|~\cite{nauty-website,mckay_14}) and added routines to it to prune graphs which are not $(k-1)$-colorable or which are not $P_t$-free. The number of graphs generated by this program keeps growing, so it cannot terminate. But it still allowed us to independently generate all $k$-critical $P_t$-free graphs up to a given number of vertices:

\begin{itemize}

\item We executed this program to generate all 4-critical $P_6$-free graphs up to 16 vertices and it indeed yielded the same 24 graphs from~\cite{CGSZ15}.

\item We executed this program to generate all 4-critical $(P_7,C_3)$-free graphs up to 16 vertices and it indeed yielded the same 7 critical graphs from Conjecture~\ref{conj:N3P7C3}.

\item We executed this program to generate all 4-critical and 4-vertex-critical $P_7$-free graphs up to 13 vertices and it indeed yielded the same graphs from~\cite{CGSZ15}.

\item We executed this program to generate all 5-critical $P_5$-free graphs up to 15 vertices and in both cases it yielded the same 90 critical graphs.

\end{itemize}

\item We modified our program to generate all $P_t$-free graphs and compared it with the known counts of $P_t$-free graphs for $t=4,5$ on the On-Line Encyclopedia of Integer Sequences~\cite{OEIS} (i.e.\ sequences A000669 and A078564).

\item We modified our program to generate all $C_u$-free graphs and compared it with the known counts of $C_u$-free graphs for $u=4,5$ on the On-Line Encyclopedia of Integer Sequences~\cite{OEIS} (i.e.\ sequences A079566 and A078566).

\item We modified our program to generate all $k$-colorable graphs and compared it with the known counts of  $k$-colorable graphs for $k=3,4$ on the On-Line Encyclopedia of Integer Sequences~\cite{OEIS} (i.e.\ sequences A076322 and A076323).

\item We tested $k$-criticality in two independent ways and both methods yielded exactly the same results:
\begin{enumerate}

\item By removing every vertex $v$ once and verifying that $G-v$ is $(k-1)$-colorable and by removing all possible non-empty sets of edges $E$ and verifying that every $G-E$ which is $\mathcal H$-free is $(k-1)$-colorable.

\item By collecting all non-$(k-1)$-colorable graphs generated by the program and only outputting the graphs which do not contain any other graph from the list as a subgraph.
\end{enumerate}

\item We determined all $k$-vertex-critical graphs in two independent ways which both yielded exactly the same results:

\begin{enumerate}
\item By modifying line~\ref{line:k-critical} of Algorithm~\ref{algo:construct} so it tests for $k$-vertex-criticality instead of $k$-criticality.

\item By recursively adding edges in all possible ways to the set of $k$-critical graphs (as long as the graphs remain $k$-vertex-critical) and  testing if the resulting graphs are $\mathcal H$-free. 

\end{enumerate}

\end{itemize}

\section*{Appendix 2: Adjacency lists}

This section contains the adjacency lists of various critical $\mathcal H$-free graphs mentioned in this article. These graphs and  the lists of the vertex-critical graphs can also be found on the \textit{House of Graphs}~\cite{hog} at \url{http://hog.grinvin.org/Critical}.

\subsection*{Adjacency lists of 4-critical $(P_7,C_3)$-free graphs}

Below are the adjacency lists of the seven 4-critical $(P_7,C_3)$-free graphs with at most 35 vertices from Conjecture~\ref{conj:N3P7C3}.

\begin{enumerate}
\item \{0 :1 4 6 8; 1 : 0 2 7 9; 2 : 1 3 6 10; 3 : 2 4 8 9; 4 : 0 3 5; 5 : 4 6 7 9 10; 6 : 0 2 5; 7 : 1 5 8; 8 : 0 3 7 10; 9 : 1 3 5; 10 : 2 5 8\}

\item \{0 :1 4 7 11; 1 : 0 2 8 9; 2 : 1 3 7 10; 3 : 2 4 9 11; 4 : 0 3 5 8; 5 : 4 6 9 10; 6 : 5 7 8 11; 7 : 0 2 6 9; 8 : 1 4 6 10; 9 : 1 3 5 7; 10 : 2 5 8 11; 11 : 0 3 6 10\}

\item \{0 :1 4 7 9 11; 1 : 0 2 8 10; 2 : 1 3 7 9; 3 : 2 4 8 11; 4 : 0 3 5 12; 5 : 4 6 9 10; 6 : 5 7 8; 7 : 0 2 6 12; 8 : 1 3 6 9 12; 9 : 0 2 5 8; 10 : 1 5 11 12; 11 : 0 3 10; 12 : 4 7 8 10\}

\item \{0 :1 4 7 9; 1 : 0 2 8 10; 2 : 1 3 7 11; 3 : 2 4 8 9; 4 : 0 3 5 10; 5 : 4 6 8 11; 6 : 5 7 9 10; 7 : 0 2 6 12; 8 : 1 3 5 12; 9 : 0 3 6 11; 10 : 1 4 6 12; 11 : 2 5 9 12; 12 : 7 8 10 11\}

\item \{0 :1 4 7 9 10; 1 : 0 2 8 11 12; 2 : 1 3 7 10 14; 3 : 2 4 8 9 11; 4 : 0 3 5 12 14; 5 : 4 6 8 10; 6 : 5 7 9 11; 7 : 0 2 6 13; 8 : 1 3 5 13; 9 : 0 3 6 12 14; 10 : 0 2 5 11 12; 11 : 1 3 6 10 14; 12 : 1 4 9 10 13; 13 : 7 8 12 14; 14 : 2 4 9 11 13\}

\item \{0 :1 4 7 10; 1 : 0 2 8 9 12; 2 : 1 3 7 11 13; 3 : 2 4 8 9 10; 4 : 0 3 5 12; 5 : 4 6 8 11; 6 : 5 7 9 13; 7 : 0 2 6 14; 8 : 1 3 5 13 14; 9 : 1 3 6 11 14; 10 : 0 3 11 13 14; 11 : 2 5 9 10 12; 12 : 1 4 11 13 14; 13 : 2 6 8 10 12; 14 : 7 8 9 10 12\}

\item \{0 :1 4 7 9 10; 1 : 0 2 8 11 12; 2 : 1 3 7 10 14; 3 : 2 4 8 9 11; 4 : 0 3 5 12 15; 5 : 4 6 8 10 14; 6 : 5 7 9 11 12; 7 : 0 2 6 13 15; 8 : 1 3 5 13 15; 9 : 0 3 6 13 14; 10 : 0 2 5 11 13; 11 : 1 3 6 10 15; 12 : 1 4 6 13 14; 13 : 7 8 9 10 12; 14 : 2 5 9 12 15; 15 : 4 7 8 11 14\}
\end{enumerate}

\subsection*{Adjacency lists of 4-critical $(P_7,C_4)$-free graphs}
Below are the adjacency lists of the seventeen 4-critical $(P_7,C_4)$-free graphs from Theorem~\ref{thm:N3P7Ck}.

\begin{enumerate}
\item \{0 : 1 2 3; 1 : 0 2 3; 2 : 0 1 3; 3 : 0 1 2\}

\item \{0 : 1 2 5; 1 : 0 3 5; 2 : 0 4 5; 3 : 1 4 5; 4 : 2 3 5; 5 : 0 1 2 3 4\}

\item \{0 : 1 2 4; 1 : 0 3 5; 2 : 0 4 6; 3 : 1 5 6; 4 : 0 2 6; 5 : 1 3 6; 6 : 2 3 4 5\}

\item \{0 : 1 2 7; 1 : 0 3 7; 2 : 0 4 7; 3 : 1 5 7; 4 : 2 6 7; 5 : 3 6 7; 6 : 4 5 7; 7 : 0 1 2 3 4 5 6\}

\item \{0 : 1 2 6; 1 : 0 3 5; 2 : 0 4 6 7; 3 : 1 5 7; 4 : 2 5 6; 5 : 1 3 4; 6 : 0 2 4 7; 7 : 2 3 6\}

\item \{0 : 1 2 7; 1 : 0 3 6 8; 2 : 0 4 7; 3 : 1 5 8; 4 : 2 6 8; 5 : 3 7 8; 6 : 1 4 8; 7 : 0 2 5; 8 : 1 3 4 5 6\}

\item \{0 : 1 2 7; 1 : 0 3 6; 2 : 0 4 7; 3 : 1 5 6 8; 4 : 2 6 7; 5 : 3 7 8; 6 : 1 3 4; 7 : 0 2 4 5 8; 8 : 3 5 7\}

\item \{0 : 1 2 8; 1 : 0 3 6 8; 2 : 0 4 8; 3 : 1 5 6; 4 : 2 6 7; 5 : 3 6 7; 6 : 1 3 4 5 7; 7 : 4 5 6; 8 : 0 1 2\}

\item \{0 : 1 2 8; 1 : 0 3 8; 2 : 0 4 6 8; 3 : 1 5 7; 4 : 2 6 9; 5 : 3 7 9; 6 : 2 4 9; 7 : 3 5 9; 8 : 0 1 2; 9 : 4 5 6 7\}

\item \{0 : 1 2 7 9; 1 : 0 3 6 9; 2 : 0 4 7; 3 : 1 5 8 9; 4 : 2 6 8; 5 : 3 7 8; 6 : 1 4 9; 7 : 0 2 5; 8 : 3 4 5; 9 : 0 1 3 6\}

\item \{0 : 1 2 7; 1 : 0 3 6; 2 : 0 4 7 9; 3 : 1 5 8; 4 : 2 6 9; 5 : 3 7 8; 6 : 1 4 9; 7 : 0 2 5 8; 8 : 3 5 7; 9 : 2 4 6\}

\item \{0 : 1 2 7; 1 : 0 3 6 8; 2 : 0 4 7; 3 : 1 5 8; 4 : 2 6 9; 5 : 3 7 8 9; 6 : 1 4 9; 7 : 0 2 5; 8 : 1 3 5; 9 : 4 5 6\}

\item \{0 : 1 2 7; 1 : 0 3 6 8; 2 : 0 4 7 9; 3 : 1 5 8; 4 : 2 6 9; 5 : 3 7 8; 6 : 1 4 9; 7 : 0 2 5; 8 : 1 3 5; 9 : 2 4 6\}

\item \{0 : 1 2 7; 1 : 0 3 6 8 9; 2 : 0 4 7; 3 : 1 5 8; 4 : 2 6 9; 5 : 3 7 8; 6 : 1 4 9; 7 : 0 2 5; 8 : 1 3 5; 9 : 1 4 6\}

\item \{0 : 1 2 7; 1 : 0 3 6 8; 2 : 0 4 7; 3 : 1 5 8; 4 : 2 6 8 9; 5 : 3 7 9; 6 : 1 4 8 9; 7 : 0 2 5; 8 : 1 3 4 6; 9 : 4 5 6\}

\item \{0 : 1 2 7; 1 : 0 3 6 8; 2 : 0 4 7 9; 3 : 1 5 8 9; 4 : 2 6 7 9; 5 : 3 7 8; 6 : 1 4 8; 7 : 0 2 4 5; 8 : 1 3 5 6; 9 : 2 3 4\}

\item \{0 : 1 2 9 10; 1 : 0 3 6 12; 2 : 0 4 7 8; 3 : 1 5 11; 4 : 2 6 11 12; 5 : 3 7 8; 6 : 1 4 12; 7 : 2 5 8; 8 : 2 5 7; 9 : 0 10 11; 10 : 0 9 11; 11 : 3 4 9 10; 12 : 1 4 6\}

\end{enumerate}

\subsection*{Adjacency lists of 4-critical $(P_7,C_5)$-free graphs}

Below are the adjacency lists of the six 4-critical $(P_7,C_5)$-free graphs from Theorem~\ref{thm:N3P7Ck}.

\begin{enumerate}
\item \{0 : 1 2 3; 1 : 0 2 3; 2 : 0 1 3; 3 : 0 1 2\}

\item \{0 : 2 3 4 5; 1 : 3 4 5 6; 2 : 0 4 5 6; 3 : 0 1 5 6; 4 : 0 1 2 6; 5 : 0 1 2 3; 6 : 1 2 3 4\}

\item \{0 : 1 6 7; 1 : 0 2 7; 2 : 1 3 7; 3 : 2 4 7; 4 : 3 5 7; 5 : 4 6 7; 6 : 0 5 7; 7 : 0 1 2 3 4 5 6\}

\item \{0 : 1 6 7; 1 : 0 2 7 8; 2 : 1 3 8; 3 : 2 4 8; 4 : 3 5 8; 5 : 4 6 8; 6 : 0 5 7 8; 7 : 0 1 6; 8 : 1 2 3 4 5 6\}

\item \{0 : 1 6 7; 1 : 0 2 7 9; 2 : 1 3 9; 3 : 2 4 8 9; 4 : 3 5 8; 5 : 4 6 8; 6 : 0 5 7; 7 : 0 1 6; 8 : 3 4 5; 9 : 1 2 3\}

\item \{0 : 1 6 7; 1 : 0 2 7 8 12; 2 : 1 3 9 12; 3 : 2 4 9; 4 : 3 5 9; 5 : 4 6 8; 6 : 0 5 7 8 12; 7 : 0 1 6; 8 : 1 5 6 10 11; 9 : 2 3 4; 10 : 8 11 12; 11 : 8 10 12; 12 : 1 2 6 10 11\}
\end{enumerate}

\subsection*{Adjacency list of smallest 4-critical $P_{12}$-free graph with girth five}

Below is the adjacency list of the smallest 4-critical $P_{12}$-free graph with girth 5 from  Theorem~\ref{thm:girth}.(b).

\begin{itemize}
\item \{0 : 1 4 8 11 17; 1 : 0 2 13 14 15; 2 : 1 3 12; 3 : 2 4 16 18 19; 4 : 0 3 5 20; 5 : 4 6 14; 6 : 5 7 12 13 17 18; 7 : 6 8 15 19; 8 : 0 7 9; 9 : 8 10 14 18 20; 10 : 9 11 13 16; 11 : 0 10 12 19; 12 : 2 6 11 20; 13 : 1 6 10; 14 : 1 5 9 19; 15 : 1 7 16 20; 16 : 3 10 15 17; 17 : 0 6 16; 18 : 3 6 9; 19 : 3 7 11 14; 20 : 4 9 12 15\}
\end{itemize}

\end{document}